\documentclass[a4paper, 12pt]{amsart}
\usepackage{graphicx}
\usepackage{amsmath,amssymb}

\newtheorem{thm}{Theorem}[section]

\newtheorem{lem}[thm]{Lemma}
\newtheorem{prop}[thm]{Proposition}

\theoremstyle{definition}

\theoremstyle{remark}
\newtheorem{rem}[thm]{Remark}
\theoremstyle{example}

\theoremstyle{conjecture}

\numberwithin{equation}{section}

\newcommand{\CC}{{\mathbb C}}

\newcommand{\NN}{{\mathbb N}}

\newcommand{\calC}{{\mathcal C}}

\newcommand{\calF}{{\mathcal F}}

\newcommand{\calM}{{\mathcal M }}

\newcommand{\calO}{{\mathcal O}}

\newcommand{\calW}{{\mathcal W}}

\begin{document}

\title[Topological structure]{Topological structure of the space of (weighted) composition operators between Fock spaces in several variables}

\author{Pham Trong Tien$^1$ \& Le Hai Khoi$^2$}%

\address{(Tien) Department of Mathematics, Mechanics and Informatics, Hanoi University of Science, VNU, 334 Nguyen Trai, Hanoi, Vietnam}%
\address{\quad \quad \quad Thang Long Institute of Mathematics and Applied Sciences, Nghiem Xuan Yem,
Hoang Mai, Hanoi, Vietnam}%
\email{phamtien@vnu.edu.vn, phamtien@mail.ru}

\address{(Khoi) Division of Mathematical Sciences, School of Physical and Mathematical Sciences, Nanyang Technological University (NTU),
637371 Singapore}%
\email{lhkhoi@ntu.edu.sg}

\thanks{$^1$  Supported in part by NAFOSTED grant No. 101.02-2017.313.\\
\indent $^2$ Supported in part by MOE's AcRF Tier 1 M4011724.110 (RG128/16)}

\subjclass[2010]{47B33, 47B38, 32A15}%

\keywords{Topological structure, Fock space, Composition operators, Weighted composition operator}%

\date{\today}%



\begin{abstract}
In this paper, we consider the topological structure problem for spaces of composition operators as well as nonzero weighted composition operators acting from a Fock space $\calF^p(\CC^n)$ to another one $\calF^q(\CC^n)$. Explicit descriptions of all (path) connected components and isolated points in these spaces are obtained. 
\end{abstract}




\maketitle

\section{Introduction}
Composition operators $C_{\varphi}$ and weighted composition operators $W_{\psi,\varphi}$ have been intensively investigated on various Banach spaces of holomorphic functions on the unit disc or the unit ball during the past several decades in different directions. 
One of the recent main problems in the study of such operators is to characterize (path) connected components and isolated points in spaces of these operators endowed with the operator norm topology. There is a huge literature in this topic: \cite{B-81, B-03, GGNS-08, M-89, IO-14, SS-90} on Hardy spaces, \cite{M-89, M-05} on Bergman spaces, \cite{HIO-05, MOZ-01,IIO-12, T-04} on spaces $H^{\infty}$ of all bounded holomorphic functions, \cite{HO-06} on Bloch spaces, etc.
On many spaces, this question is difficult and not yet solved completely.

Recently, much progress was made in the study of (weighted) composition operators on Fock spaces (see, for instance, \cite{CMS-03, CIK-10, T-14, HK-16-1, U-07}). One of the main differences between operators $C_{\varphi}$ and $W_{\psi,\varphi}$ on Fock spaces and those on the above-mentioned spaces of holomorphic functions on the unit disc or the unit ball is the lack of bounded holomorphic functions in the Fock space setting. In fact, entire functions $\varphi$ that induce bounded operators $C_{\varphi}$ and $W_{\psi,\varphi}$ are quite restrictive, in details, they are only affine functions. Thanks to this, some difficult problems were completely solved on Fock spaces. 
In particular, concerning the topological structure, the path connected components and isolated points in the space of composition operators on the Hilbert Fock space in several variables were characterized in \cite{D-15}.
Later, in \cite{TK-17} the authors obtained complete descriptions of all (path) connected components and isolated points in not only  the space of composition operators but also the space of nonzero weighted composition operators between different general Fock spaces in one variable. 

The aim of this paper is to develop the study related to the topological structure in \cite{TK-17} for the case of several variables. Roughly speaking, our main result is to give complete answers to all important questions of the topological structure problem for both spaces of composition operators and nonzero weighted composition operators in the Fock space context. It should be noted that the techniques used in \cite{D-15} for Hilbert Fock spaces cannot be applied to this paper for general ones. Also the techniques in several variables are much more complicated than in one variable. 

The paper is organized as follows. In Section 2 we recall some preliminaries results on general Fock spaces $\calF^p(\CC^n)$ and (weighted) composition operators between different Fock spaces. 
Section 3 is devoted to the space $\calC(\calF^p(\CC^n), \calF^q(\CC^n))$ of composition operators acting from a Fock space $\calF^p(\CC^n)$ to another one $\calF^q(\CC^n)$. We prove that if $0 < q < p < \infty$, then the space $\calC(\calF^p(\CC^n), \calF^q(\CC^n))$ is path connected (Theorem \ref{thm-tp-1}). In the case $0 < p \leq q < \infty$, we completely determine all (path) connected components and isolated points in $\calC(\calF^p(\CC^n), \calF^q(\CC^n))$ (Theorem \ref{thm-tp-co}).
The study of the space $\calC_w(\calF^p(\CC^n), \calF^q(\CC^n))$  of nonzero weighted composition operators acting from a Fock space $\calF^p(\CC^n)$ to another one $\calF^q(\CC^n)$ is more complicated and carried out in Section 4. In Theorem \ref{thm-tp-qp}, we show that the space $\calC_w(\calF^p(\CC^n), \calF^q(\CC^n))$ is also path connected when $0 < q < p < \infty$, while all (path) connected components of $\calC_w(\calF^p(\CC^n), \calF^q(\CC^n))$ when $0 < p \leq q < \infty$ are characterized in Theorem \ref{thm-main-tp}.

It should be noted that the key technique in this paper is to study composition operators $C_{\varphi}$ and weighted composition operators $W_{\psi,\varphi}$ via the operators $C_{\widetilde{\varphi}}$ and $W_{\widetilde{\psi}, \widetilde{\varphi}}$, which are induced by the so-called \textit{normalizations} $\widetilde{\varphi}$ of $\varphi$ and $(\widetilde{\psi}, \widetilde{\varphi})$ of $(\psi, \varphi)$, respectively.

\section{Preliminaries}
Recall that for a number $p \in (0,\infty)$, the Fock space $\calF^p(\CC^n)$ consists of all entire functions $f$ on $\CC^n$ for which
$$
\|f\|_{n,p} = \left( \left(\dfrac{p}{2\pi}\right)^n \int_{\CC^n} |f(z)|^p e^{-\frac{p|z|^2}{2}}dA(z) \right)^{\frac{1}{p}} < \infty,
$$
where $dA$ is the Lebesgue measure on $\CC^n$. 
It is well known that $\calF^p(\CC^n)$ with $1 \leq p < \infty$ is a Banach space, while for $0 < p < 1$, $\calF^p(\CC^n)$ is a complete metric space with the distance $d(f,g) = \|f-g\|^p_{n, p}$.

For each $w \in \CC^n$, we define the functions
$$
K_w(z) = e^{\langle z,w \rangle} \text{ and } k_w(z) = e^{\langle z,w \rangle - \frac{|w|^2}{2}}, \quad z \in \CC^n,
$$
where $\langle z, w \rangle = z_1 \overline{w_1} + \cdots + z_n \overline{w_n}$ and $|w|^2 = \langle w, w \rangle$. 
Then $\|k_w\|_{n, p} = 1$ for all $w \in \CC^n$ and $0< p < \infty$, and $k_w$ converges to $0$ in the space $\calO(\CC^n)$ as $|w| \to \infty$, where $\calO(\CC^n)$ is the space of all entire functions on $\CC^n$ with the usual compact open topology.

We give some notation and auxiliary results which will be used throughout the paper.

For each point $z = (z_1,...,z_n) \in \CC^n$ and $0 \leq s \leq n$, we define 
$$
z_{[s]} = 
\begin{cases}
\emptyset, \quad \quad \quad \ \ \text{ if } s = 0 \\
(z_1,...,z_s), \text{ if } s \neq 0,
\end{cases}
\text{and  }
z'_{[s]} = 
\begin{cases}
(z_{s+1},...,z_n), \text{ if } s \neq n\\
\emptyset, \quad \quad \quad \quad \ \ \text{ if } s = n,
\end{cases}
$$
by convention that 
$ |z_{[0]}| = 0$ and $|z'_{[n]}| = 0$.

For each  $z = (z_1,...,z_n)\in \mathbb C^n$ and $1 \leq i \leq n$, we put
$$
z'_i = 
\begin{cases} 
  (z_2,...,z_n), \qquad \qquad \qquad \text{if } i = 1,\\
  (z_1,...,z_{i-1}, z_{i+1},...,z_n), \ \text{ if } 2 \leq i \leq n-1,\\
  (z_1,...,z_{n-1}), \qquad \qquad \quad \text{ if } i = n,
\end{cases}
$$
and, briefly, write $z = (z_i,z'_i)$.

For an $n \times n$ diagonal matrix $A$ and $0 < s < n$, we denote by $A_{[s]}$ the principal submatrix of $A$ with diagonal entries $a_{ii}, i = 1,...,s$, and by $A'_{[s]}$ the principal submatrix of $A$ with diagonal entries $a_{ii}, i = s+1,...,n$. 

The following lemmas can be found in \cite[Section 2]{TK-17-1}.

\begin{lem}\label{lem-F}
Let $p \in (0,\infty)$, $b = (b_1,...,b_n)$ be a point in $\CC^n$, and $f \in \calF^p(\CC^n)$. For each $0 < s < n$ the following statements are true: 
\begin{itemize}
\item[(i)] The function $f(b_{[s]},\cdot) \in \calF^p(\CC^{n-s})$ and 
$$
\|f(b_{[s]},\cdot)\|_{n-s, p}\; e^{-\frac{\left|b_{[s]}\right|^2}{2}} \leq \|f\|_{n, p}.
$$ 
\item[(ii)] The function $f(\cdot,b'_{[s]}) \in \calF^p(\CC^s)$ and 
$$
\|f(\cdot,b'_{[s]})\|_{s, p}\; e^{-\frac{\left|b'_{[s]}\right|^2}{2}} \leq \|f\|_{n, p}.
$$ 
\end{itemize}
\end{lem}

\begin{lem} \label{lem-F1}
Let $p \in (0, \infty)$ and $1 \leq i \leq n$ be given. Then for each function $f \in \calF^p(\CC^{n})$ and $z \in \CC^n$, the following inequalities hold:
\begin{itemize}
\item[(i)] 
$$
|f(z)|e^{-\frac{|z|^2}{2}} \leq \|f\|_{n, p}.
$$
\item[(ii)]
$$
\left|\dfrac{\partial f}{\partial z_i}(z)\right| \leq e^2 (1+|z_i|)e^{\frac{|z|^2}{2}} \|f\|_{n,p}.
$$
\end{itemize}
\end{lem}
\begin{proof}
(i) was proved in \cite[Lemma 2.2]{TK-17-1}.

(ii) For every point $z'_{i} \in \mathbb C^{n-1}$ fixed, by Lemma \ref{lem-F}, the function $f(\cdot, z'_i) \in \calF^p(\mathbb C)$. Then, by \cite[Lemma 2.1]{TK-17}, for every $z_i \in \mathbb C$,
$$
 \left| \dfrac{\partial f}{\partial z_i}(z_i, z'_i)\right| \leq e^2 (1 + |z_i|) e^{\frac{|z_i|^2}{2}} \|f(\cdot,z'_i)\|_{1,p}.
$$
This and Lemma \ref{lem-F} imply the desired inequality. 
\end{proof}

\begin{lem}\label{lem-Fpq}
For every $ 0 < p < q < \infty$, $\calF^p(\CC^n) \subset \calF^q(\CC^n)$ and the inclusion is continuous. Moreover,
$$
\|f\|_{n, q} \leq \left(  \dfrac{q}{p} \right) ^{\frac{n}{q}} \|f\|_{n, p}, \ \forall f \in \calF^p(\CC^n).
$$ 
\end{lem}

Let  $\psi: \CC^n \to \CC$ be a nonzero entire function and $\varphi: \CC^n \to \CC^n$ a holomorphic mapping. The \textit{weighted composition operator} $W_{\psi,\varphi}$ induced by $\psi$ and $\varphi$ is defined as follows $W_{\psi,\varphi}f = \psi \cdot (f \circ \varphi)$. When the function $\psi$ is identically $1$, the operator $W_{\psi,\varphi}$ reduces to the \textit{composition operator} $C_{\varphi}$. 
As in \cite{T-14, TK-17}, we define the following quantities
$$
m_z(\psi,\varphi) = |\psi(z)|e^{\frac{|\varphi(z)|^2 - |z|^2}{2}},\ z \in \CC^n,
$$
and
$$
m(\psi,\varphi) = \sup_{z \in \CC^n} m_z(\psi,\varphi).
$$
In \cite[Section 3]{TK-17-1} it was shown that bounded weighted composition operators from a Fock space $\calF^p(\CC^n)$ to another one $\calF^q(\CC^n)$ can be induced only by nonzero entire functions $\psi \in \calF^q(\CC^n)$ and such mappings $\varphi(z) = Az + b$ with some $n \times n$ matrix $A$, $\|A\| \leq 1$ and $n \times 1$ vector $b$. In particular, boundedness and compactness of composition operators $C_{\varphi}: \calF^p(\CC^n) \to \calF^q(\CC^n)$ were characterized in terms of the matrix $A$. These characterizations will be used in the sequel and for the reader's convenience we state them in the following theorems. The proofs can be founded in \cite[Corollaries 3.11 and 3.14]{TK-17-1}.

\begin{thm}\label{thm-co}
Let $0< p \leq q < \infty$ and $\varphi: \CC^n \to \CC^n$ a holomorphic mapping. The following statements are true.
\begin{itemize}
\item[(a)] The operator $C_{\varphi}: \calF^p(\CC^n) \to \calF^q(\CC^n)$ is bounded if and only if $\varphi(z) = Az + b$, where $A$ is an $n \times n$ matrix and  $b$ is an $n \times 1$ vector such that $\|A\| \leq 1$ and $\langle A\zeta, b \rangle =0$ for every $\zeta$ in $\CC^n$ with $|A\zeta| = |\zeta|$.\
\item[(b)] The operator $C_{\varphi}: \calF^p(\CC^n) \to \calF^q(\CC^n)$ is compact if and only if $\varphi(z) = Az + b$, where $A$ is an $n \times n$ matrix with $\|A\| < 1$ and  $b$ is an $n \times 1$ vector.
\end{itemize}
\end{thm}

\begin{thm}\label{thm-co-1}
Let $0< q < p < \infty$ and $\varphi: \CC^n \to \CC^n$ a holomorphic mapping. The following conditions are equivalent:
\begin{itemize}
\item[(i)] The operator $C_{\varphi}: \calF^p(\CC^n) \to \calF^q(\CC^n)$ is bounded.
\item[(ii)] The operator $C_{\varphi}: \calF^p(\CC^n) \to \calF^q(\CC^n)$ is compact.
\item[(iii)] $\varphi(z) = Az + b$, where $A$ is an $n \times n$ matrix with $\|A\| < 1$ and $b$ is an $n \times 1$ vector.
\end{itemize}
\end{thm}

The criteria for boundedness and compactness of weighted composition operators $W_{\psi,\varphi}: \calF^p(\CC^n) \to \calF^q(\CC^n)$ were obtained in terms of the so-called \textit{normalization} $(\widetilde{\psi}, \widetilde{\varphi})$ of $(\psi,\varphi)$ (see \cite[Theorems 3.8, 3.9 and 3.12]{TK-17-1}). The normalization $(\widetilde{\psi}, \widetilde{\varphi})$ also plays an important role in the current paper. We recall this notation based on the following singular value decomposition of the matrix $A$ (see also \cite[Theorem 2.6.3]{HR-90}).

\begin{lem}
If $A$ is an $n \times n$ matrix of rank $s$, then $A$ can be written as $A = V \widetilde{A} U$, where $V, U$ are $n \times  n$ unitary matrices, and $\widetilde{A}$ is a diagonal matrix $(\widetilde{a}_{ij})$ with $\widetilde{a}_{11} \geq \widetilde{a}_{22} \geq ... \geq \widetilde{a}_{ss} \geq \widetilde{a}_{s+1, s+1} = ... = \widetilde{a}_{nn} = 0$. The $\widetilde{a}_{ii}$ are the non-negative square roots of the eigenvalues of $AA^*$; if we require that they are listed in decreasing order, then $\widetilde{A}$ is uniquely determined from $A$.
\end{lem}

Let $\calW_q$ be the set of all pairs $(\psi,\varphi)$ consisting of a nonzero entire function $\psi$ in $\calF^q(\CC^n)$ and a mapping $\varphi(z) = Az + b$ with an $n \times n$ matrix $A$ satisfying $\|A\| \leq 1$ and an $n \times 1$ vector $b$. 

We denote by $\mathcal V_{q,s}$ the subset of $\calW_q$ consisting of all pairs $(\psi,\varphi)$ in $\calW_q$ with $\varphi(z) = Az + b$, where $A$ is a diagonal matrix $(a_{ij})$ of $\textnormal{rank} A = s > 0$ and 
$$
1 \geq a_{11} \geq a_{22} \geq ... \geq a_{ss} \geq a_{s+1, s+1} = ... = a_{nn} = 0.
$$
Note that for each $(\psi,\varphi)$ in $\mathcal V_{q,s}$ and $f \in \calO(\CC^n)$, we have
\begin{align}\label{eq-normf}
\|W_{\psi,\varphi}f\|_{n, q} =& \left( \left( \dfrac{q}{2\pi} \right)^n \int_{\CC^n} |\psi(z)|^q |f(\varphi(z))|^q e^{-\frac{q|z|^2}{2}}dA(z)\right)^{\frac{1}{q}}\\ \nonumber
=& \left( \left( \dfrac{q}{2\pi} \right)^s \int_{\CC^s}  |f(\varphi(z))|^q e^{-\frac{q\left|z_{[s]}\right|^2}{2}} \|\psi(z_{[s]},\cdot)\|^q_{n-s,q} dA(z_{[s]})\right)^{\frac{1}{q}}.
\end{align}
In view of this, boundedness and compactness of weighted composition operators $W_{\psi,\varphi}$ induced by a pair $(\psi,\varphi)$ in $\mathcal V_{q,s}$, are characterized by the following quantities:
\begin{align*}
\ell_{z_{[s]}}(\psi,\varphi) = e^{\frac{|\varphi(z)|^2 - \left|z_{[s]}\right|^2}{2}} \|\psi(z_{[s]},.)\|_{n-s, q}, \ \ z_{[s]} \in \CC^s,
\end{align*}
and
$$
\ell(\psi,\varphi) = \sup_{z_{[s]} \in \CC^s} \ell_{z_{[s]}}(\psi, \varphi),
$$
where we consider $\|\psi(z_{[s]},.)\|_{n-s, q} = |\psi(z)|$ if $s = n$, and in this case $\ell_z(\psi,\varphi) = m_z(\psi,\varphi)$.

For each pair $(\psi, \varphi)$ in $\calW_q$ with $\varphi(z) = Az + b$ and $\textnormal{rank} A = s > 0$ and the singular value decomposition $V \widetilde{A} U$ of $A$, we define a new pair $(\widetilde{\psi}, \widetilde{\varphi})$ as follows:
$$
\widetilde{\psi}(z) = \psi(U^*z) \text{ and } \widetilde{\varphi}(z) = \widetilde{A} z + \widetilde{b},\ \widetilde{b} = V^*b, \ z \in \CC^n.
$$
We call $(\widetilde{\psi}, \widetilde{\varphi})$ a \textit{normalization} of $(\psi, \varphi)$ with respect to the singular value decomposition $A = V \widetilde{A} U$ (briefly, \textit{normalization} of $(\psi, \varphi)$). In the case $\psi$ is identically $1$, $\widetilde{\varphi}$ is also called a \textit{normalization} of $\varphi$.
Note that $(\widetilde{\psi}, \widetilde{\varphi})$ belongs to $\mathcal V_{q, s}$.

We give the following auxiliary lemma which will be used in the sequel.

\begin{lem}\label{lem-psiphi}
Let $(\psi,\varphi)$ be a pair in $\calW_q$ with $\varphi(z) = Az + b$ and $(\widetilde{\psi}, \widetilde{\varphi})$ its normalization. If $m(\psi,\varphi) < \infty$ and $\|A\| = 1$, then
$$
\widetilde{\psi}(z) = e^{-\left\langle z_{[j]}, \widetilde{b}_{[j]}\right\rangle} \widetilde{\psi}_*(z'_{[j]}), \ z \in \mathbb C^n,
$$
where $j = \max\{i: \widetilde{a}_{ii} = 1\}$ and $\widetilde{\psi}_*$ is a nonzero entire function of $z'_{[j]}$ on $\mathbb C^{n-j}$.
In particular, if $\psi$ is identically $1$, then $\widetilde{b}_{i} = 0$ for every $i \leq j$.
\end{lem}
\begin{proof}
By \cite[Lemma 3.5]{TK-17-1}, we have $m(\widetilde{\psi}, \widetilde{\varphi}) = m(\psi,\varphi) < \infty$. Then, for each $z \in \mathbb C^n$, 
\begin{align*}
\infty > m(\widetilde{\psi}, \widetilde{\varphi}) &\geq |\widetilde{\psi}(z)|e^{\frac{|\widetilde{\varphi}(z)|^2 - |z|^2}{2}} \\
& = \left|\widetilde{\psi}(z) e^{\left\langle z_{[j]}, \widetilde{b}_{[j]} \right\rangle}  \right|e^{\frac{\left|\widetilde{b}_{[j]}\right|^2}{2}}e^{\frac{\left|\widetilde{A}'_{[j]}z'_{[j]} + \widetilde{b}'_{[j]}\right|^2 - \left|z'_{[j]}\right|^2}{2}}.
\end{align*}
Thus, for each $z'_{[j]} \in \mathbb C^{n-j}$ fixed, the entire function $\widetilde{\psi}(z_{[j]}, z'_{[j]}) e^{\left\langle z_{[j]}, \widetilde{b}_{[j]} \right\rangle}$ is bounded on $\mathbb C^j$, and hence, $\widetilde{\psi}(z_{[j]}, z'_{[j]}) e^{\left\langle z_{[j]}, \widetilde{b}_{[j]} \right\rangle} = \widetilde{\psi}_*(z'_{[j]})$ for every $z \in \mathbb C^n$.
From this the conclusions follow.
\end{proof}

Next, for simplicity, if $\varphi(z) \equiv b$, then we write $C_b$ instead of $C_{\varphi}$; and 
in the case $\varphi(z) = Az$ with an $n \times n$ matrix $A$, we denote the composition operator $C_{\varphi}$ by $C_A$.
Obviously, if $U$ is a unitary matrix, then $C_U$ is invertible on every Fock space $\calF^p(\CC^n)$ with $(C_U)^{-1} = C_{U^*}$ and $\|C_Uf\|_{n, p} = \|f\|_{n, p}$ for all $f \in \calF^p(\CC^n)$. From this and the definition of $(\widetilde{\psi}, \widetilde{\varphi})$ and \cite[Proposition 3.4]{TK-17-1} it follows that
\begin{equation}\label{eq-w-nor}
W_{\psi,\varphi} = C_U W_{\widetilde{\psi}, \widetilde{\varphi}} C_V \text{ and } W_{\widetilde{\psi}, \widetilde{\varphi}} = C_{U^*} W_{\psi,\varphi} C_{V^*}.
 \end{equation}

\section{The space of composition operators}

In this section we study path connected components in the space $\calC(\calF^p(\mathbb C^n), \calF^q(\mathbb C^n))$ of all composition operators acting from $\calF^p(\mathbb C^n)$ to $\calF^q(\mathbb C^n)$ under the operator norm topology. For two operators $C_{\varphi}$ and $C_{\phi}$ in $\calC(\calF^p(\mathbb C^n), \calF^q(\mathbb C^n))$, we  write $C_{\varphi} \sim C_{\phi}$ in $\calC(\calF^p(\mathbb C^n), \calF^q(\mathbb C^n))$ if $C_{\varphi}$ and $C_{\phi}$ belong to the same path connected component of the space $\calC(\calF^p(\mathbb C^n), \calF^q(\mathbb C^n))$.

Firstly we investigate the set of all compact composition operators from $\calF^p(\mathbb C^n)$ to $\calF^q(\mathbb C^n)$, denoted by $\calC_0(\calF^p(\mathbb C^n), \calF^q(\mathbb C^n))$. By Theorems \ref{thm-co} and \ref{thm-co-1}, we get 
$$
\calC_0(\calF^p(\mathbb C^n), \calF^q(\mathbb C^n)) = \left\{C_{\varphi}: \varphi(z) = Az + b, \|A\| < 1, b \in \mathbb C^n \right\}.
$$

\begin{prop}\label{prop-tp-com}
Let $p, q \in (0,\infty)$ be given. The set $\calC_0(\calF^p(\mathbb C^n), \calF^q(\mathbb C^n))$ of all compact composition operators acting from $\calF^p(\mathbb C^n)$ to $\calF^q(\mathbb C^n)$ is path connected in the space $\calC(\calF^p(\mathbb C^n),\calF^q(\mathbb C^n))$.
\end{prop}

\begin{proof}
We divide the proof into three steps.

\textbf{Step 1.} We show that if the operator $C_{\varphi}:\calF^p(\mathbb C^n) \to \calF^q(\mathbb C^n)$ is compact, then $C_{\varphi} \sim C_{\varphi(0)}$ in $\calC(\calF^p(\mathbb C^n),\calF^q(\mathbb C^n))$ via a path in $\calC_0(\calF^p(\mathbb C^n), \calF^q(\mathbb C^n))$.
Let $\varphi(z) = Az + b$, where $A$ is an $n \times n$ matrix with $\|A\| < 1$ and $b$ is an $n \times 1$ vector.

If $A = 0$, then the assertion is trivial. Suppose that $0 < \|A\| < 1$ and $\textnormal{rank}A = s$. For each $t \in [0,1]$, put $\varphi_t(z) = \varphi(tz) = tAz + b, \ z \in \mathbb C^n$. By Theorems \ref{thm-co} and \ref{thm-co-1}, all operators $C_{\varphi_t}$ with $t \in [0,1]$ are compact from  $\calF^p(\mathbb C^n)$ to $\calF^q(\mathbb C^n)$, i. e. $C_{\varphi_t} \in \calC_0(\calF^p(\mathbb C^n), \calF^q(\mathbb C^n))$ for all $t \in [0,1]$. Moreover, $C_{\varphi} = C_{\varphi_1}$ and $C_{\varphi(0)} = C_{\varphi_0}$. 

We now prove that the map 
$$
[0,1] \longrightarrow \calC(\calF^p(\mathbb C^n), \calF^q(\mathbb C^n)), \ t \longmapsto C_{\varphi_t},
$$
is continuous, that is, 
$\|C_{\varphi_t} - C_{\varphi_{t_0}}\| \to 0$ as $t \to t_0$ for every $t_0 \in [0,1]$.

Let $\widetilde{\varphi}(z) = \widetilde{A}z + \widetilde{b}$ be the normalization of $\varphi$, where the singular value decomposition of $A$ is $V \widetilde{A} U$ and $\widetilde{b} = V^*b$.
Then $\displaystyle \widetilde{\varphi_t \ }(z) = t\widetilde{A}z + \widetilde{b}$, and by \eqref{eq-w-nor}, $C_{\varphi_t} = C_{U} C_{\widetilde{\varphi_t}} C_{V}$ and $C_{\widetilde{\varphi_t}} = C_{U^*}C_{\varphi_t}C_{V^*}$ for all $t \in [0,1]$.
Thus, $\|C_{\varphi_t} - C_{\varphi_{t_0}}\|  = \|C_{\widetilde{\varphi_t}} - C_{\widetilde{\varphi_{t_0}}}\|$ for all $t, t_0 \in [0,1].$

Fix $t_0 \in [0,1]$. For every $t \in [0,1]$ and $f \in \calF^p(\mathbb C^n)$ with $\|f\|_{n,p} \leq 1$, using \eqref{eq-normf} and the fact that for each $z_{[s]} \in \CC^s$, there is a number $\tau = \tau(z_{[s]})$ in between $t_0$ and $t$ such that
\begin{align*}
f(\widetilde{\varphi_t}(z)) - f(\widetilde{\varphi_{t_0}}(z)) &= 
f\big(t\widetilde{A}_{[s]}z_{[s]} + \widetilde{b}_{[s]}, \widetilde{b}'_{[s]}\big) - f\big(t_0\widetilde{A}_{[s]}z_{[s]} + \widetilde{b}_{[s]}, \widetilde{b}'_{[s]}\big) \\
&= (t - t_0)\sum_{i=1}^s \widetilde{a}_{ii}z_i \dfrac{\partial f}{\partial z_i}\big(\tau \widetilde{A}_{[s]}z_{[s]} + \widetilde{b}_{[s]}, \widetilde{b}'_{[s]}\big),
\end{align*}
we obtain
\begin{align*}
& \|C_{\widetilde{\varphi_t}}f - C_{\widetilde{\varphi_{t_0}}}f\|_{n,q}^q \\
= & \left( \dfrac{q}{2 \pi} \right)^s \int_{\mathbb C^s} \left|f\big(t\widetilde{A}_{[s]}z_{[s]} + \widetilde{b}_{[s]}, \widetilde{b}'_{[s]}\big) - f\big(t_0\widetilde{A}_{[s]}z_{[s]} + \widetilde{b}_{[s]}, \widetilde{b}'_{[s]}\big)\right|^q e^{-\frac{q\left|z_{[s]}\right|^2}{2}}\; dA(z_{[s]}) \\
= & |t - t_0|^q \left( \dfrac{q}{2 \pi} \right)^s \int_{\mathbb C^s} \left|\sum_{i=1}^s \widetilde{a}_{ii}z_i \dfrac{\partial f}{\partial z_i}\big(\tau \widetilde{A}_{[s]}z_{[s]} + \widetilde{b}_{[s]}, \widetilde{b}'_{[s]}\big) \right|^q e^{-\frac{q\left|z_{[s]}\right|^2}{2}}\; dA(z_{[s]}).
\end{align*}
From this and Lemma \ref{lem-F1}(ii), for some constant $C>0$ satisfying
$$
(x_1 + ... + x_s)^q \leq C^q (x_1^q + ... + x_s^q), \ \forall x_1,...,x_s \geq 0,
$$
we get that for every $t \in [0,1]$ and $f \in \calF^p(\mathbb C^n)$ with $\|f\|_{n,p} \leq 1$,
\begin{align*}
& \|C_{\widetilde{\varphi_t}}f - C_{\widetilde{\varphi_{t_0}}}f\|_{n,q}^q \\
\leq &\; C^q |t-t_0|^q \left( \dfrac{q}{2 \pi} \right)^s \int_{\mathbb C^s} \sum_{i=1}^s |\widetilde{a}_{ii}z_i|^q \left|\dfrac{\partial f}{\partial z_i}\big(\tau \widetilde{A}_{[s]}z_{[s]} + \widetilde{b}_{[s]}, \widetilde{b}'_{[s]}\big) \right|^q e^{-\frac{q\left|z_{[s]}\right|^2}{2}}\; dA(z_{[s]})\\
\leq &\; C^q e^{2q} |t-t_0|^q  \left( \dfrac{q}{2 \pi} \right)^s \\
\times & \int_{\mathbb C^s} \sum_{i=1}^s |\widetilde{a}_{ii}z_i|^q \big(1 + \big|\tau \widetilde{a}_{ii}z_i + \widetilde{b}_i\big|\big)^q  e^{\frac{q\left|\big(\tau \widetilde{A}_{[s]}z_{[s]} + \widetilde{b}_{[s]}, \widetilde{b}'_{[s]}\big)\right|^2}{2}} \|f\|^q_{n,p} e^{-\frac{q\left|z_{[s]}\right|^2}{2}}\; dA(z_{[s]})\\
\leq &\; C^q e^{2q} |t-t_0|^q  \left( \dfrac{q}{2 \pi} \right)^s \\
\times & \int_{\mathbb C^s} \sum_{i=1}^s |\widetilde{a}_{ii}z_i|^q \big(1 + |\widetilde{a}_{ii}z_i| + |\widetilde{b}_i|\big)^q e^{\frac{q \big(\left|\widetilde{A}_{[s]}z_{[s]}\right| + \left|\widetilde{b}_{[s]}\right|\big)^2 + q\left|\widetilde{b}'_{[s]}\right|^2}{2}} e^{-\frac{q\left|z_{[s]}\right|^2}{2}}\; dA(z_{[s]})\\
= & \; M^q |t-t_0|^q,
\end{align*}
where
\begin{align*}
& M^q = C^q e^{2q}  \left( \dfrac{q}{2 \pi} \right)^s \\
\times & \int_{\mathbb C^s} \sum_{i=1}^s |\widetilde{a}_{ii}z_i|^q \big(1 + |\widetilde{a}_{ii}z_i| + |\widetilde{b}_i|\big)^q e^{\frac{q \big(\left|\widetilde{A}_{[s]}z_{[s]}\right| + \left|\widetilde{b}_{[s]}\right|\big)^2 + q\left|\widetilde{b}'_{[s]}\right|^2}{2}} e^{-\frac{q\left|z_{[s]}\right|^2}{2}}\; dA(z_{[s]}) < \infty, 
\end{align*}
since $\|\widetilde{A}_{[s]}\| < 1$.

Consequently,
$$
\|C_{\varphi_t} - C_{\varphi_{t_0}}\|  = \|C_{\widetilde{\varphi_t}} - C_{\widetilde{\varphi_{t_0}}}\| \leq M |t-t_0|,\ \forall t, t_0 \in [0,1]. 
$$
This implies that 
$\displaystyle \lim_{t \to t_0}\|C_{\varphi_t} - C_{\varphi_{t_0}}\| = 0$ for all $t_0 \in [0,1]$.

\textbf{Step 2.} We prove that for every $\alpha, \beta \in \mathbb C^n$, the operators $C_{\alpha} \sim C_{\beta}$ in $\calC(\calF^p(\mathbb C^n),\calF^q(\mathbb C^n))$ via a path in $\calC_0(\calF^p(\mathbb C^n), \calF^q(\mathbb C^n))$.

For each $t \in [0,1]$, put $\gamma_t = (1-t)\alpha + t\beta$. Then $C_{\alpha} = C_{\gamma_0}$ and $C_{\beta} = C_{\gamma_1}$, and all operators $C_{\gamma_t}, t \in [0,1],$ are compact from $\calF^p(\mathbb C^n)$ to $\calF^q(\mathbb C^n)$, i. e. $C_{\gamma_t} \in \calC_0(\calF^p(\mathbb C^n), \calF^q(\mathbb C^n))$ for all $t \in [0,1]$.

We show that the map 
$$
[0,1] \longrightarrow \calC(\calF^p(\mathbb C^n), \calF^q(\mathbb C^n)),\  t \longmapsto C_{\gamma_t},
$$
is continuous.
Fix an arbitrary number $t_0 \in [0,1]$. For each $t \in [0,1]$ and $f \in \calF^p(\mathbb C^n)$ with $\|f\|_{n,p} \leq 1$, using Lemma \ref{lem-F1}(ii), for some number $\tau$ in between $t_0$ and $t$, we have
\begin{align*}
& \|C_{\gamma_t}f - C_{\gamma_{t_0}}f\|_{n,q} = |f(\gamma_t) - f(\gamma_{t_0})| \|1\|_{n,q} \\
 = &\; |t - t_0| \left| \sum_{i=1}^n \dfrac{\partial f(\gamma_{\tau})}{\partial z_i} (\beta_i - \alpha_i) \right|  \leq |t-t_0| \sum_{i=1}^n |\beta_i - \alpha_i| \left| \dfrac{\partial f(\gamma_{\tau})}{\partial z_i}\right| \\
 \leq & \; e^2 |t-t_0| \sum_{i=1}^n |\beta_i - \alpha_i| (1 + |(1-\tau)\alpha_i + \tau \beta_i|) e^{\frac{|\gamma_{\tau}|^2}{2}} \|f\|_{n,p}  \\
 \leq &\; e^2 |t-t_0| \sum_{i=1}^n |\beta_i - \alpha_i| (1 + |\alpha_i| + |\beta_i|) e^{\frac{(|\alpha| + |\beta|)^2}{2}}.
\end{align*}
From this it follows that
$$
\|C_{\gamma_t} - C_{\gamma_{t_0}}\| \leq e^2 |t-t_0| \sum_{i=1}^n |\beta_i - \alpha_i| (1 + |\alpha_i| + |\beta_i|) e^{\frac{(|\alpha| + |\beta|)^2}{2}}, \ \forall t, t_0 \in [0,1].
$$
Therefore,
$ \displaystyle \lim_{t \to t_0}\|C_{\gamma_t} - C_{\gamma_{t_0}}\| = 0.$

\textbf{Step 3.} Let $C_{\varphi}$ and $C_{\phi}$ be two compact composition operators from $\calF^p(\mathbb C^n)$ to $\calF^q(\mathbb C^n)$. By Steps 1 and 2, $C_{\varphi} \sim C_{\varphi(0)} \sim C_{\phi(0)} \sim C_{\phi}$ in $\calC(\calF^p(\mathbb C^n),\calF^q(\mathbb C^n))$ via the paths in the set $\calC_0(\calF^p(\mathbb C^n),\calF^q(\mathbb C^n))$.

From this the assertion follows.
\end{proof}

From Theorem \ref{thm-co-1} and Proposition \ref{prop-tp-com} we immediately get the following result.

\begin{thm}\label{thm-tp-1}
If $0 < q < p < \infty$, then the space $\calC(\calF^p(\mathbb C^n), \calF^q(\mathbb C^n))$ is path connected.
\end{thm}

Now we study the case when $0 < p \leq q < \infty$. Let denote by $E_n$ the set of all $n \times n$ matrices whose  norm is not greater than $1$. 
We say that two matrices $A$ and $D$ in $E_n$ are \textit{equivalent} and, briefly, write $A \sim D$ if $A\xi = D \xi$ for all $\xi \in \mathbb C^n$ with $|A\xi| = |\xi|$ or $|D\xi| = |\xi|$. Obviously, the relation $A \sim D$ is an equivalence relation on $E_n$. Let $[E_n]$ be the set of all equivalence classes induced by this relation $A \sim D$.

For a class $[A] \in [E_n]$, we denote by $\calC(n,p,q,[A])$ the set of all bounded composition operators $C_{\varphi}$ from $\calF^p(\mathbb C^n)$ to $\calF^q(\mathbb C^n)$ induced by $\varphi(z) = Az + b$ with $A \in [A]$. Then, by Theorem \ref{thm-co},
$$
\calC(n,p,q,[A]) = \left\{C_{\varphi}: \varphi(z) = Az + b, A \in [A], \langle A\xi, b \rangle = 0 \text{ if } |A\xi| = |\xi|\right\}.
$$
We can verify that $[0] = \{A \in E_n: \|A\| < 1\}$. Indeed, if $\|A\| < 1$, then $|A\xi| < |\xi|$ for all nonzero vectors $\xi \in \mathbb C^n$, hence $A$ is equivalent to the zero matrix. Otherwise, if $\|A\| = 1$, then there is a nonzero vector $\xi \in \mathbb C^n$ such that $|A\xi| = |\xi|$, but obviously, $A\xi \neq 0\xi$, i.e., $A \not \sim 0$. From this and Theorem \ref{thm-co} and Proposition \ref{prop-tp-com}, it follows that $\calC(n,p,q,[0]) = \calC_0(\calF^p(\CC^n), \calF^q(\CC^n))$ is a path connected set in $\calC(\calF^p(\mathbb C^n), \calF^q(\mathbb C^n))$. 

In view of this, we will show that this statement is also true for all sets $\calC(n,p,q,[A])$, i.e. all sets $\calC(n,p,q,[A])$ are path connected in $\calC(\calF^p(\mathbb C^n), \calF^q(\mathbb C^n))$. To do this, we need the following results.

\begin{lem}\label{lem-sim}
For every nonzero equivalence class $[A] \in [E_n]$, there exist a number $j \in \NN$ and a pair of $n \times n$ unitary matrices $(V, U)$ such that the class $[A]$ consists of all the following matrices
\begin{equation}\label{eq-ma} 
A = V\begin{pmatrix}
I_j & 0 \\
0 & G
\end{pmatrix}U, 
\end{equation}
where $I_j$ is the $j \times j$ unit matrix and $G$ is an $(n-j) \times (n-j)$ matrix with $\|G\| < 1$.
\end{lem}
\begin{proof}
Fix an arbitrary matrix $A^0 \in [A]$ and suppose that the singular value decomposition of $A^0$ is $V \widetilde{A^0}U$ with 
$
\widetilde{A^0} = \begin{pmatrix}
I_j & 0 \\
0 & G^0
\end{pmatrix},
$
where $j = \max\{i: \widetilde{a^0_{ii}} = 1\}$ and $G^0$ is a diagonal $(n-j) \times (n-j)$ matrix with non-negative diagonal elements in the decreasing order and $\|G^0\| < 1$.

By the proof of \cite[Lemma 2.4]{D-15}, every matrix $A \in [A]$, i.e. $A \sim A^0$, can be represented as in \eqref{eq-ma}.

Conversely, we can easily see that every matrix $A$ in the form \eqref{eq-ma} is equivalent to $A^0$, i.e. $A \in [A]$. Indeed, for every $\xi \in \CC^n$,
$ |A\xi| = |\xi|$ or $|A^0\xi| = |\xi|$ if and only if $(U\xi)'_{[j]} = 0$; and hence, 
$$
A\xi = V \begin{pmatrix}
(U\xi)_{[j]}\\
0
\end{pmatrix}
= A^0\xi
$$
for every $\xi \in \CC^n$ satisfying $ |A\xi| = |\xi|$ or $|A^0\xi| = |\xi|$. 
\end{proof}

\begin{rem}\label{rem-sim}
While the number $j$ in Lemma \ref{lem-sim} is uniquely determined for each class $[A] \in [E_n]\setminus \{[0]\}$, the pair of unitary matrices $(V,U)$ is not unique. However, all these pairs are closely related in the following sense. Suppose that $(\widehat{V},\widehat{U})$ is another pair. Then 
$A^0 = \widehat{V}\begin{pmatrix}
I_j & 0 \\
0 & G
\end{pmatrix}\widehat{U},$
where $ G$ is some $(n-j) \times (n-j)$ matrix and $\|G\| < 1$.

Let $G = V_1 \widetilde{G} U_1$ be the singular value decomposition of $G$ and put 
\begin{equation}\label{eq-uv0}
V_0 = \begin{pmatrix}
I_j & 0 \\
0 & V_1
\end{pmatrix}
\text{ and }
U_0 = \begin{pmatrix}
I_j & 0 \\
0 & U_1
\end{pmatrix}.
\end{equation}
We have
$$
A^0 = \widehat{V}V_0\begin{pmatrix}
I_j & 0 \\
0 & \widetilde{G}
\end{pmatrix}U_0\widehat{U}.
$$
From this it follows that $\widetilde{G} = G^0$, which means that $(\widehat{V}V_0, U_0\widehat{U})$ is the other pair of unitary factors of the singular value decomposition 
$A^0 = (\widehat{V}V_0) \widetilde{A^0} (U_0\widehat{U})$. Then, by \cite[Theorem~2.6.5]{HR-90}, there are $n_1 \times n_1$ unitary matrix $H_1$,..., $n_d \times n_d$ unitary matrix $H_d$ and $(n-s) \times (n-s)$ unitary matrices $W_1$ and $W_2$ such that
\begin{equation}\label{eq-uv}
\widehat{V}V_0 = V (H_1 \oplus ... \oplus H_d \oplus W_1) \text{ and } U_0\widehat{U} = (H_1^* \oplus ... \oplus H_d^* \oplus W_2^*)U
\end{equation}
where, for each $1 \leq i \leq d$, $n_i$ is the multiplicity of the distinct positive singular value $\sigma_i$ of $A^0$ and $\sigma_1 > \sigma_2 > ... > \sigma_d$ and $\textnormal{rank}A^0 = s$. In this case $\sigma_1 = 1$ and $n_1 = j$.
\end{rem}

\begin{prop}\label{prop-est-norm}
Let $0 < p \leq q < \infty$ and $C_{\varphi}$ and $C_{\phi}$ be two composition operators in $\calC(\calF^p(\mathbb C^n), \calF^q(\mathbb C^n))$ with $\varphi(z) = Az + b$ and $\phi(z) = Dz + e$. If $A \not\sim D$ then 
$$
\|C_{\varphi} - C_{\phi}\| \geq \dfrac{1}{2}.
$$
\end{prop}
\begin{proof}
Since $A \not\sim D$, there exists a point $\xi \in \mathbb C^n$ such that $A\xi \neq D\xi$ and $|A\xi| = |\xi| $ or $|D\xi| = |\xi|$. Without loss of generality, assume that $|A\xi| = |\xi|$.

For all $z, w \in \mathbb C^n$, by Lemma \ref{lem-F1}(i) and $\|k_w\|_{n,p} = 1$, we have
\begin{align*}
\|C_{\varphi} - C_{\phi}\| & \geq \|C_{\varphi}k_w - C_{\phi}k_w\|_{n,q} 
 \geq \left| e^{\langle \varphi(z), w \rangle} - e^{\langle \phi(z), w \rangle} \right| e^{-\frac{|z|^2 + |w|^2}{2}}.
\end{align*}
In particular, if $w = \varphi(z)$ or $w = \phi(z)$, then
\begin{align*}
\|C_{\varphi} - C_{\phi}\| & \geq e^{\frac{|\varphi(z)|^2 - |z|^2}{2}} - e^{\frac{|\phi(z)|^2 - |z|^2}{2}} e^{-\frac{|\varphi(z) - \phi(z)|^2}{2}},
\end{align*}
or, respectively,
\begin{align*}
\|C_{\varphi} - C_{\phi}\| & \geq e^{\frac{|\phi(z)|^2 - |z|^2}{2}} - e^{\frac{|\varphi(z)|^2 - |z|^2}{2}} e^{-\frac{|\phi(z) - \varphi(z)|^2}{2}},
\end{align*}
for all $z \in \mathbb C^n$. Thus
\begin{align*}
2\|C_{\varphi} - C_{\phi}\| & \geq \left( e^{\frac{|\varphi(z)|^2 - |z|^2}{2}} + e^{\frac{|\phi(z)|^2 - |z|^2}{2}}\right) \left( 1 - e^{-\frac{|\varphi(z) - \phi(z)|^2}{2}}\right) \\
& \geq e^{\frac{|\varphi(z)|^2 - |z|^2}{2}}\left( 1 - e^{-\frac{|\varphi(z) - \phi(z)|^2}{2}}\right), \ \forall z \in \mathbb C^n.
\end{align*}

On the other hand, since $|A\xi| = |\xi|$, by Theorem \ref{thm-co}(a), $\langle A\xi, b \rangle = 0$. Hence, for all $\lambda \in \mathbb C$,
$$
|\varphi(\lambda \xi)|^2 - |\lambda \xi|^2 = |\lambda A\xi + b|^2 - |\lambda \xi|^2 = |b|^2
$$ 
and 
$$
|\varphi(\lambda \xi) - \phi(\lambda \xi)|^2 \geq (|\lambda| |A\xi - D\xi| - |b-e|)^2 \to +\infty \text{ as } \ |\lambda| \to +\infty,
$$
since $A\xi \neq D\xi$.
Consequently, with $z = \lambda \xi, \lambda \in \CC$, we have
\begin{align*}
2\|C_{\varphi} - C_{\phi}\| & \geq e^{\frac{|b|^2}{2}}\left( 1 - e^{-\frac{|\varphi(\lambda \xi) - \phi(\lambda \xi)|^2}{2}}\right) \to e^{\frac{|b|^2}{2}} \text{ as }  \ |\lambda| \to +\infty.
\end{align*}
From this the desired inequality follows.
\end{proof}

\begin{prop}\label{prop-varphiA}
Let $0 < p \leq q < \infty$ and $C_{\varphi}$ be a composition operator in $\calC(\calF^p(\mathbb C^n), \calF^q(\mathbb C^n))$ induced by $\varphi(z) = Az + b$ with $\|A\| = 1$. Then the operators $C_{\varphi}$ and $C_A$ belong to the same path connected component of  $\calC(\calF^p(\mathbb C^n), \calF^q(\mathbb C^n))$.
\end{prop}
\begin{proof}

For each $t \in [0,1]$, put $\varphi_t(z) = Az + tb, \ z \in \mathbb C^n$. Then, by Theorem \ref{thm-co}, all operators $C_{\varphi_t}, t \in [0,1]$, are bounded from  $\calF^p(\mathbb C^n)$ to $\calF^q(\mathbb C^n)$ and $C_{\varphi_1} = C_{\varphi}$ and $C_{\varphi_0} = C_A$. Thus, we need to show that the map 
$$
[0,1] \longrightarrow \calC(\calF^p(\mathbb C^n), \calF^q(\mathbb C^n)), t \longmapsto C_{\varphi_t},
$$
is continuous, that is, 
$ \|C_{\varphi_t} - C_{\varphi_{t_0}}\| \to 0$ as $t \to t_0$ for all $t_0 \in [0,1]$.

Let $\widetilde{\varphi}(z) = \widetilde{A}z + \widetilde{b}$ be the normalization of $\varphi$, where the singular value decomposition of $A$ is $V \widetilde{A} U$ and $\widetilde{b} = V^*b$.
 Then $\displaystyle \widetilde{\varphi_t \ }(z) = \widetilde{A}z + t\widetilde{b}$ and by \eqref{eq-w-nor}, $C_{\varphi_t} = C_{U} C_{\widetilde{\varphi_t}} C_{V}$ and $C_{\widetilde{\varphi_t}} = C_{U^*}C_{\varphi_t}C_{V^*}$ for all $t \in [0,1]$.
Thus,
$ \|C_{\varphi_t} - C_{\varphi_{t_0}}\|  = \|C_{\widetilde{\varphi_t}} - C_{\widetilde{\varphi_{t_0}}}\|$ for all $t, t_0 \in [0,1]$.

Put $\textnormal{rank}A = s$ and $j = \max \{i: \widetilde{a}_{ii} = 1\}$. Then, by Lemma \ref{lem-psiphi}, $\widetilde{b}_{i} = 0$ for all $i \leq j$.

Fix $t_0 \in [0,1]$. For every $t \in [0,1]$ and $f \in \calF^p(\mathbb C^n)$ with $\|f\|_{n,p} \leq 1$, using the fact that for each $z \in \CC^n$, there is a number $\tau =\tau(z)$ in between $t$ and $t_0$ such that
\begin{align*}
f(\widetilde{A}z + t\widetilde{b}) - f(\widetilde{A}z + t_0 \widetilde{b}) & = 
f\big(z_{[j]}, \big(\widetilde{A}z + t\widetilde{b}\big)'_{[j]}\big) - f\big(z_{[j]}, \big(\widetilde{A}z + t_0 \widetilde{b}\big)'_{[j]}\big) \\
& = (t - t_0) \sum_{i = j+1 }^n \widetilde{b}_i \dfrac{\partial f}{\partial z_i}\big(z_{[j]}, \big(\widetilde{A}z + \tau \widetilde{b}\big)'_{[j]}\big),
\end{align*}
we obtain
\begin{align*}
& \|C_{\widetilde{\varphi_t}}f - C_{\widetilde{\varphi_{t_0}}}f\|_{n,q}^q \\
= & \left( \dfrac{q}{2 \pi} \right)^n \int_{\mathbb C^n} \left|f(\widetilde{A}z + t\widetilde{b}) - f(\widetilde{A}z + t_0 \widetilde{b})\right|^q e^{-\frac{q|z|^2}{2}}\; dA(z) \\
\leq &\; C^q |t-t_0|^q \left( \dfrac{q}{2 \pi} \right)^n \int_{\mathbb C^n}  \sum_{i = j+1 }^n \left|\widetilde{b}_i \dfrac{\partial f}{\partial z_i}\big(z_{[j]}, \big(\widetilde{A}z + \tau \widetilde{b}\big)'_{[j]}\big) \right|^q e^{-\frac{q|z|^2}{2}}\; dA(z),
\end{align*} 
for some constant $C>0$.

Moreover, for every $i = j+1,...,n$ and $z \in \CC^n$ fixed, applying Lemma \ref{lem-F1}(ii) to the function 
$ f\big(z_{[j]}, \cdot\big)$ in $\calF^p(\mathbb C^{n-j})$, hence, in $\calF^q(\mathbb C^{n-j})$, we get
\begin{align*}
\left|\dfrac{\partial f}{\partial z_i}\big(z_{[j]}, \big(\widetilde{A}z + \tau \widetilde{b}\big)'_{[j]}\big) \right| \leq &\; e^2 \big(1 + |\widetilde{a}_{ii} z_i + \tau \widetilde{b}_i|\big) e^{\frac{\left|\big(\widetilde{A}z + \tau \widetilde{b}\big)'_{[j]}\right|^2}{2}} \|f(z_{[j]}, \cdot)\|_{n-j, q} \\
\leq &\; e^2 \big(1 + |\widetilde{a}_{ii} z_i| + |\widetilde{b}_i|\big) e^{\frac{\left(\left|\widetilde{A}'_{[j]}z'_{[j]}\right| + \left|\widetilde{b}'_{[j]}\right|\right)^2}{2}} \|f(z_{[j]}, \cdot)\|_{n-j, q}.
\end{align*}
From this and Lemma \ref{lem-Fpq}, it follows that for every $i = j+1,...,n$,
\begin{align*}
&\left(\dfrac{q}{2\pi}\right)^{j} \int_{\mathbb C^j}\left|\dfrac{\partial f}{\partial z_i}\big(z_{[j]}, \big(\widetilde{A}z + \tau \widetilde{b}\big)'_{[j]}\big) \right|^q e^{-\frac{q\left|z_{[j]}\right|^2}{2}}dA(z_{[j]}) \\
\leq & \; e^{2q} \big(1 + |\widetilde{a}_{ii} z_i| + |\widetilde{b}_i|\big)^q e^{\frac{q\left(\left|\widetilde{A}'_{[j]}z'_{[j]}\right| + \left|\widetilde{b}'_{[j]}\right|\right)^2}{2}} \\
& \qquad \qquad  \qquad \qquad \times \left(\dfrac{q}{2\pi}\right)^{j} \int_{\mathbb C^{j}} \|f(z_{[j]}, \cdot)\|_{n-j, q}^q  e^{-\frac{q\left|z_{[j]}\right|^2}{2}} dA(z_{[j]}) \\
= & \ e^{2q} \big(1 + |\widetilde{a}_{ii} z_i| + |\widetilde{b}_i|\big)^q e^{\frac{q\left(\left|\widetilde{A}'_{[j]}z'_{[j]}\right| + \left|\widetilde{b}'_{[j]}\right|\right)^2}{2}} \|f\|_{n, q}^q \\
\leq & \ e^{2q} \big(1 + |\widetilde{a}_{ii} z_i| + |\widetilde{b}_i|\big)^q e^{\frac{q\left(\left|\widetilde{A}'_{[j]}z'_{[j]}\right| + \left|\widetilde{b}'_{[j]}\right|\right)^2}{2}} \left( \dfrac{q}{p} \right)^n \|f\|_{n,p}^q \\
\leq & \ e^{2q} \left( \dfrac{q}{p} \right)^n \big(1 + |\widetilde{a}_{ii} z_i| + |\widetilde{b}_i|\big)^q e^{\frac{q\left(\left|\widetilde{A}'_{[j]}z'_{[j]}\right| + \left|\widetilde{b}'_{[j]}\right|\right)^2}{2}}. 
\end{align*}
Therefore, for every $t \in [0,1]$ and $f \in \calF^p(\CC^n)$ with $\|f\|_{n,p} \leq 1$, we get 
\begin{align*}
& \|C_{\widetilde{\varphi_t}}f - C_{\widetilde{\varphi_{t_0}}}f\|_{n,q}^q \\
\leq &\ C^q |t-t_0|^q \left( \dfrac{q}{2 \pi} \right)^n \int_{\mathbb C^n}  \sum_{i = j+1 }^n \left|\widetilde{b}_i \dfrac{\partial f}{\partial z_i}\big(z_{[j]}, \big(\widetilde{A}z + \tau \widetilde{b}\big)'_{[j]}\big) \right|^q e^{-\frac{q|z|^2}{2}}\; dA(z)\\
\leq &\ C^q e^{2q} |t-t_0|^q \left( \dfrac{q}{p} \right)^n \left( \dfrac{q}{2 \pi} \right)^{n-j} \\
& \ \ \ \times \int_{\mathbb C^{n-j}}  \sum_{i = j+1 }^n |\widetilde{b}_i|^q \big(1 + |\widetilde{a}_{ii} z_i| + |\widetilde{b}_i|\big)^q e^{\frac{q\left(\left|\widetilde{A}'_{[j]}z'_{[j]}\right| + \left|\widetilde{b}'_{[j]}\right|\right)^2}{2}} e^{-\frac{q\left|z'_{[j]}\right|^2}{2}}\; dA(z'_{[j]})\\
= & \ M^q |t-t_0|^q,
\end{align*} 
where 
\begin{align*}
& M^q = C^q e^{2q} \left( \dfrac{q}{p} \right)^n  \left( \dfrac{q}{2 \pi} \right)^{n-j}\\
\times & \sum_{i = j+1 }^n |\widetilde{b}_i |^q \int_{\mathbb C^{n-j}} \big(1 + |\widetilde{a}_{ii} z_i| + |\widetilde{b}_i|\big)^q e^{\frac{q \big(\left|\widetilde{A}'_{[j]}z'_{[j]}\right| + \left|\widetilde{b}'_{[j]}\right|\big)^2}{2}} e^{-\frac{q\left|z'_{[j]}\right|^2}{2}}\; dA(z'_{[j]}) < \infty,
\end{align*}
since $\|\widetilde{A}'_{[j]}\| < 1$. 

Consequently, 
$\|C_{\varphi_t} - C_{\varphi_{t_0}}\|  = \|C_{\widetilde{\varphi_t}} - C_{\widetilde{\varphi_{t_0}}}\| \leq M|t-t_0|$ for all $t \in [0,1]$,
which completes the proof.
\end{proof}

\begin{prop}\label{prop-tp-0}
Let $0 < p \leq q < \infty$ and $1 \leq j \leq n-1$ and $G$ be an arbitrary $(n-j) \times (n-j)$ matrix with $\|G\| < 1$. Put
$$
A = \begin{pmatrix}
I_j & 0 \\
0 & 0
\end{pmatrix}
\text{ and }
D = \begin{pmatrix}
I_j & 0 \\
0 & G
\end{pmatrix}
$$
where $I_j$ is the $j \times j$ unit matrix. Then $C_A$ and $C_D$ are in the same path connected component of $\calC(\calF^p(\mathbb C^n), \calF^q(\mathbb C^n))$.
\end{prop}
\begin{proof}
Put $\varphi_t(z) = (1-t) Az + tDz$ for each $t \in [0,1]$. Then $C_{\varphi_0} = C_A$ and $C_{\varphi_1} = C_D$, and by Theorem \ref{thm-co}, the operators $C_{\varphi_t}$ are bounded from $\calF^p(\CC^n)$ to $\calF^q(\CC^n)$ for all $t \in [0,1]$.
We show that the map 
$$
[0,1] \longrightarrow \calC(\calF^p(\mathbb C^n), \calF^q(\mathbb C^n)), t \longmapsto C_{\varphi_t},
$$
is continuous.

Let $G = V_1 \widetilde{G} U_1$ be the singular value decomposition of $G$ and put
$$
V = \begin{pmatrix}
I_j & 0 \\
0 & V_1
\end{pmatrix}
\text{ and }
U = \begin{pmatrix}
I_j & 0 \\
0 & U_1
\end{pmatrix}.
$$
Then $V$ and $U$ are $n \times n$ unitary matrices and put
$$
\widetilde{A} = V^* A U^* = \begin{pmatrix}
I_j & 0 \\
0 & 0
\end{pmatrix}
\text{ and }
\widetilde{D} = V^* D U^* = \begin{pmatrix}
I_j & 0 \\
0 & \widetilde{G}
\end{pmatrix}.
$$
Obviously, for each $t \in [0,1]$, $\widetilde{\varphi_{t\;}}(z) = (1 - t) \widetilde{A}z + t\widetilde{D}z$ and by \eqref{eq-w-nor}, $C_{\varphi_t} = C_{U}C_{\widetilde{\varphi_t}}C_{V}$ and $C_{\widetilde{\varphi_t}} = C_{U^*}C_{\varphi_t}C_{V^*}$.
It implies that
$\|C_{\widetilde{\varphi_t}} - C_{\widetilde{\varphi_{t_0}}}\| = \|C_{\varphi_t} - C_{\varphi_{t_0}}\|$ for all $t, t_0 \in [0,1]$.

Fix $t_0 \in [0,1]$. For every $t \in [0,1]$ and $f \in \calF^p(\mathbb C^n)$ with $\|f\|_{n,p} \leq 1$, using the fact that for each $z \in \CC^n$, there is a number $\tau = \tau(z)$ in between $t$ and $t_0$ such that
\begin{align*}
f(\widetilde{\varphi_t\; }(z)) - f(\widetilde{\varphi_{t_0}}(z)) & = 
f\big(z_{[j]}, t\widetilde{G}z'_{[j]}\big) - f\big(z_{[j]}, t_0\widetilde{G}z'_{[j]}\big)\\
& = (t - t_0) \sum_{i = j+1 }^n \widetilde{g}_{ii}z_i \dfrac{\partial f}{\partial z_i}\big(z_{[j]}, \tau \widetilde{G}z'_{[j]}\big),
\end{align*}
we get 
\begin{align*}
& \|C_{\widetilde{\varphi_t}}f - C_{\widetilde{\varphi_{t_0}}}f\|_{n,q}^q \\
= & \left( \dfrac{q}{2 \pi} \right)^n \int_{\mathbb C^n} \big|f(\widetilde{\varphi_t\; }(z)) - f(\widetilde{\varphi_{t_0}}(z))\big|^q e^{-\frac{q|z|^2}{2}}\; dA(z) \\
\leq &\; C^q |t-t_0|^q \left( \dfrac{q}{2 \pi} \right)^n \int_{\mathbb C^n}  \sum_{i = j+1 }^n \left|\widetilde{g}_{ii}z_i \dfrac{\partial f}{\partial z_i}\big(z_{[j]}, \tau \widetilde{G}z'_{[j]}\big) \right|^q e^{-\frac{q|z|^2}{2}}\; dA(z),
\end{align*}
for some constant $C> 0$.

Similarly to the proof of Proposition \ref{prop-varphiA}, for every $i = j+1,...,n$ and $z \in \CC^n$ fixed, using Lemma \ref{lem-F1}(ii), we have
\begin{align*}
\left|\dfrac{\partial f}{\partial z_i}\big(z_{[j]}, \tau \widetilde{G}z'_{[j]}\big) \right| &\leq e^2 \big(1 + |\tau \widetilde{g}_{ii} z_i|\big) e^{\frac{\big| \tau \widetilde{G}z'_{[j]} \big|^2 }{2}} \left\|f\big(z_{[j]}, \cdot\big)\right\|_{n-j, q}\\
& \leq e^2 \big(1 + |\widetilde{g}_{ii} z_i|\big) e^{\frac{\big|\widetilde{G}z'_{[j]} \big|^2 }{2}} \left\|f\big(z_{[j]}, \cdot\big)\right\|_{n-j, q}.
\end{align*}
From this and Lemma \ref{lem-Fpq}, for every $i = j+1, ..., n$, we obtain
\begin{align*}
&\left(\dfrac{q}{2\pi}\right)^{j} \int_{\mathbb C^j}\left|\dfrac{\partial f}{\partial z_i}\big(z_{[j]}, \tau \widetilde{G}z'_{[j]}\big) \right|^q e^{-\frac{q\left|z_{[j]}\right|^2}{2}}dA(z_{[j]}) \\
\leq & \; e^{2q} \big(1 + |\widetilde{g}_{ii} z_i|\big)^q e^{\frac{q\big|\widetilde{G}z'_{[j]} \big|^2}{2}} \left(\dfrac{q}{2\pi}\right)^{j} \int_{\mathbb C^{j}} \|f\big(z_{[j]}, \cdot\big)\|_{n-j, q}^q e^{-\frac{q\left|z_{[j]}\right|^2}{2}} dA(z_{[j]}) \\
= &\; e^{2q} \big(1 + |\widetilde{g}_{ii} z_i|\big)^q e^{\frac{q\big|\widetilde{G}z'_{[j]} \big|^2}{2}} \left\|f\right\|_{n,q}^q \\
\leq &\; e^{2q} \big(1 + |\widetilde{g}_{ii} z_i|\big)^q e^{\frac{q\left|\widetilde{G}z'_{[j]}\right|^2}{2}} \left( \dfrac{q}{p} \right)^n \|f\|_{n,p}^q \\
\leq &\; e^{2q} \left( \dfrac{q}{p} \right)^n \big(1 + |\widetilde{g}_{ii} z_i|\big)^q e^{\frac{q\left|\widetilde{G}z'_{[j]}\right|^2}{2}}. 
\end{align*}
Therefore, for every $t \in [0,1]$ and $f \in \calF^p(\CC^n)$ with $\|f\|_{n,p} \leq 1$, we get
\begin{align*}
& \|C_{\widetilde{\varphi_t}}f - C_{\widetilde{\varphi_{t_0}}}f\|_{n,q}^q \\
\leq &\; C^q |t-t_0|^q \left( \dfrac{q}{2 \pi} \right)^n \int_{\mathbb C^n}  \sum_{i = j+1 }^n \left|\widetilde{g}_{ii}z_i \dfrac{\partial f}{\partial z_i}\big(z_{[j]}, \tau \widetilde{G}z'_{[j]}\big) \right|^q e^{-\frac{q|z|^2}{2}}\; dA(z)\\
\leq &\; C^q e^{2q} |t-t_0|^q \left( \dfrac{q}{p} \right)^n \left( \dfrac{q}{2 \pi} \right)^{n-j} \\
& \ \ \ \times \int_{\mathbb C^{n-j}}  \sum_{i = j+1 }^n |\widetilde{g}_{ii}z_i|^q \big(1 + |\widetilde{g}_{ii} z_i|\big)^q e^{\frac{q\left|\widetilde{G}z'_{[j]}\right|^2}{2}} e^{-\frac{q\left|z'_{[j]}\right|^2}{2}}\; dA(z'_{[j]})\\
= &\; M^q |t-t_0|^q,
\end{align*}
where
\begin{align*}
M^q &= C^q e^{2q} \left( \dfrac{q}{p} \right)^n  \left( \dfrac{q}{2 \pi} \right)^{n-j}\\
\times & \sum_{i = j+1 }^n \int_{\mathbb C^{n-j}} |\widetilde{g}_{ii}z_i|^q \big(1 + |\widetilde{g}_{ii} z_i|\big)^q e^{\frac{q\left|\widetilde{G}z'_{[j]}\right|^2}{2}} e^{-\frac{q\left|z'_{[j]}\right|^2}{2}}\; dA(z'_{[j]}) < \infty,
\end{align*}
since $\|\widetilde{G}\| < 1$.

Consequently, 
$ \|C_{\varphi_t} - C_{\varphi_{t_0}}\|  = \|C_{\widetilde{\varphi_t}} - C_{\widetilde{\varphi_{t_0}}}\| \leq M|t-t_0|$ for all $t \in [0,1]$, 
which completes the proof.
\end{proof}

From these auxiliary results we can get necessary and sufficient conditions under which two composition operators belong to the same path connected component in the space $\calC(\calF^p(\mathbb C^n), \calF^q(\mathbb C^n))$.

\begin{thm}\label{thm-tp-eq-co}
Let $0 < p \leq q < \infty$ be given. Suppose that $C_{\varphi}$ and $C_{\phi}$ are two composition operators in $\calC(\calF^p(\mathbb C^n), \calF^q(\mathbb C^n))$ induced by $\varphi(z) = Az + b$ and $\phi(z) = Dz + e$. Then  $C_{\varphi}$ and $C_{\phi}$ are in the same path connected component of $\calC(\calF^p(\mathbb C^n), \calF^q(\mathbb C^n))$ if and only if $A \sim D$.
\end{thm}
\begin{proof}
\textbf{Necessity.} Suppose that $C_{\varphi} \sim C_{\phi}$ in $\calC(\calF^p(\mathbb C^n), \calF^q(\mathbb C^n))$. This means that there is a continuous path  in $\calC(\calF^p(\mathbb C^n), \calF^q(\mathbb C^n))$ connecting $C_{\varphi}$ and $C_{\phi}$. Then we can find a finite sequence of composition operators $(C_{\varphi_i})_{i=1}^m$ in $\calC(\calF^p(\mathbb C^n), \calF^q(\mathbb C^n))$ induced by $\varphi_i(z) = A^i z + b^i$ so that 
$$
C_{\varphi_1} = C_{\varphi} \text{ and } C_{\varphi_m} = C_{\phi} \text{ and } \|C_{\varphi_i} - C_{\varphi_{i+1}}\| < \frac{1}{2},
$$
for all  $i = 1, 2,..., m-1$.
From this and Proposition \ref{prop-est-norm}, it follows that $A^i \sim A^{i+1}$ for all $i = 1, 2,..., m-1$. Thus, $A \sim D$.

\textbf{Sufficiency.} Suppose now $A \sim D$. If one of two matrices $A$ and $D$ has norm less than $1$, then so does the other. Hence, $C_{\varphi}$ and $C_{\phi}$ are compact. From this and Proposition \ref{prop-tp-com} the assertion follows. 

Now consider the case $\|A\| = \|D\| = 1$. Then, by Lemma \ref{lem-sim}, there exist $n \times n$ unitary matrices $U$ and $V$ such that
$$
A = V \begin{pmatrix}
I_j & 0 \\
0 & G
\end{pmatrix}
U
\text{ and }
D = V \begin{pmatrix}
I_j & 0 \\
0 & G_1
\end{pmatrix}
U,
$$
where $G$ and $G_1$ are $(n-j) \times (n-j)$ matrices with $\|G\| < 1$ and $\|G_1\| < 1$.
Put 
$$
A_1 = \begin{pmatrix}
I_j & 0 \\
0 & G
\end{pmatrix}
,
D_1 = \begin{pmatrix}
I_j & 0 \\
0 & G_1
\end{pmatrix}
\text{ and } 
A_0 = \begin{pmatrix}
I_j & 0 \\
0 & 0
\end{pmatrix}.
$$
Then, by Proposition \ref{prop-tp-0}, $C_{A_1} \sim C_{A_0} \sim C_{D_1}$ in $\calC(\calF^p(\mathbb C^n), \calF^q(\mathbb C^n))$.
 
On the other hand, $C_A = C_U C_{A_1} C_V$ and $C_D = C_U C_{D_1} C_V$, hence, it is easy to see that $C_{A} \sim C_{D}$ in $\calC(\calF^p(\mathbb C^n), \calF^q(\mathbb C^n))$.
Indeed, if $\mathcal P_t$ is a continuous path connecting $C_{A_1}$ and $C_{D_1}$ in $\calC(\calF^p(\mathbb C^n), \calF^q(\mathbb C^n))$, then $C_U \mathcal P_t C_V$ is a continuous path in $\calC(\calF^p(\mathbb C^n), \calF^q(\mathbb C^n))$ connecting the operators $C_{A}$ and $C_{D}$.

Moreover, by Proposition \ref{prop-varphiA}, we get that $C_{\varphi} \sim C_A$ and $C_{\phi} \sim C_D$ in $\calC(\calF^p(\mathbb C^n), \calF^q(\mathbb C^n))$.

Following these statements we can complete the proof.
\end{proof}

Finally, we get the following complete description of (path) connected components in the space $\calC(\calF^p(\mathbb C^n), \calF^q(\mathbb C^n))$.

\begin{thm}\label{thm-tp-co}
Let $0 < p \leq q < \infty$. Then the space $\calC(\calF^p(\mathbb C^n), \calF^q(\mathbb C^n))$ has the following (path) connected components:
$$
\calC(\calF^p(\mathbb C^n), \calF^q(\mathbb C^n)) = \bigcup_{[A] \in [E_n]} \calC(n, p, q, [A]).
$$
In particular, if $A$ is a unitary matrix, then $C_A$ is an isolated point of $\calC(\calF^p(\mathbb C^n), \calF^q(\mathbb C^n))$.
\end{thm}
\begin{proof}
Obviously, all sets $\calC(n,p,q,[A])$ with $[A] \in [E_n]$ are disjoint. Moreover, by Proposition \ref{prop-est-norm}, these sets $\calC(n,p,q,[A])$ are all open in the space $\calC(\calF^p(\mathbb C^n), \calF^q(\mathbb C^n))$.

On the other hand, by Theorem \ref{thm-tp-eq-co}, two arbitrary composition operators in $\calC(n,p,q,[A])$ belong to the same path connected component of $\calC(\calF^p(\mathbb C^n), \calF^q(\mathbb C^n))$. 

Consequently, every set $\calC(n,p,q,[A])$ with  $[A] \in [E_n]$ is a path connected component and, simultaneously, a connected component of the space $\calC(\calF^p(\mathbb C^n), \calF^q(\mathbb C^n))$. 

Moreover, if $A$ is a unitary matrix in $\calM_n$, then $[A] = \{A\}$, and hence, the set $\calC(n,p,q,[A])$ contains only the operator $C_A$. That is, $C_A$ is an isolated point in $\calC(\calF^p(\mathbb C^n), \calF^q(\mathbb C^n))$. 
\end{proof}

\begin{rem}
In the case $p = q =2$, our Theorem \ref{thm-tp-co} clarifies the corresponding main results in \cite{D-15} and gives an explicit description of all path connected components and isolated points in the space $\calC(\calF^2(\mathbb C^n), \calF^2(\mathbb C^n))$.
\end{rem}

\section{The space of nonzero weighted composition operators}

In this section we explicitly describe path connected components in the space $\calC_w(\calF^p(\mathbb C^n), \calF^q(\mathbb C^n))$ of all nonzero weighted composition operators acting from $\calF^p(\mathbb C^n)$ to $\calF^q(\mathbb C^n)$ under the operator norm topology. 
For two operators $W_{\psi, \varphi}$ and $W_{\chi,\phi}$ in $\calC_w(\calF^p(\mathbb C^n), \calF^q(\mathbb C^n))$, we write $W_{\psi, \varphi} \sim W_{\chi,\phi}$ in $\calC_w(\calF^p(\mathbb C^n), \calF^q(\mathbb C^n))$ if $W_{\psi, \varphi}$ and $W_{\chi,\phi}$ belong to the same path connected component of $\calC_w(\calF^p(\mathbb C^n), \calF^q(\mathbb C^n))$. 

For $p, q \in (0,\infty)$ and $\varphi(z) = Az + b$ with $A \in E_n$ we denote by $\calF(n,p,q,\varphi)$ the set of all nonzero functions $\psi \in \calF^q(\mathbb C^n)$ such that the operator $W_{\psi,\varphi}: \calF^p(\mathbb C^n) \to \calF^q(\mathbb C^n)$ is bounded. Then
$$
\calC_{w}(\calF^p(\mathbb C^n), \calF^q(\mathbb C^n))= \{W_{\psi,\varphi}: \varphi(z) = Az + b, A \in E_n, \psi \in \calF(n,p,q,\varphi)\}.
$$

Similarly to \cite[Lemma 4.8]{TK-17}, we have the following lemma.

\begin{lem}\label{lem-tp-psi}
Let $p, q \in (0,\infty)$, $\varphi(z) = Az + b$ with $\|A\| \leq 1$ and $\psi, \chi \in \calF(n,p,q,\varphi)$. Then the operators $W_{\psi,\varphi}$ and $W_{\chi, \varphi}$ belong to the same path connected component of $\calC_w(\calF^p(\mathbb C^n), \calF^q(\mathbb C^n))$.
\end{lem}

\begin{proof}
We can easily show that there exists a complex valued continuous function $\alpha(t)$ on $[0,1]$ such that $\alpha(0) = 0$, $\alpha(1) = 1$ and $u_t = (1 - \alpha(t))\psi_1 + \alpha(t) \psi_2$ are all nonzero functions in $\calF(n, p, q, \varphi)$. 

Then, the operators $W_{\psi,\varphi}$ and $W_{\chi, \varphi}$ are in the same path connected component of $\calC_w(\calF^p(\mathbb C^n), \calF^q(\mathbb C^n))$ via the continuous path $W_{u_t,\varphi}$.
\end{proof}

As in Section 3, we firstly study the following subset 
$$
\calC_{w,0}(\calF^p(\mathbb C^n), \calF^q(\mathbb C^n)) = \{W_{\psi,\varphi}: \varphi(z) = Az + b, \|A\| < 1, \psi \in \calF(n,p,q,\varphi)\}
$$
of the space $\calC_{w}(\calF^p(\mathbb C^n), \calF^q(\mathbb C^n))$.
\begin{prop}\label{prop-tp-wco-0}
 Let $p, q \in (0, \infty)$ be given. The set $\calC_{w,0}(\calF^p(\mathbb C^n), \calF^q(\mathbb C^n))$ is path connected in the space $\calC_w(\calF^p(\mathbb C^n), \calF^q(\mathbb C^n))$.
\end{prop}
\begin{proof}
Let $W_{\psi,\varphi}$ and $W_{\chi, \phi}$ be two weighted composition operators in $\calC_{w,0}(\calF^p(\mathbb C^n), \calF^q(\mathbb C^n))$. Then, by Theorems \ref{thm-co} and \ref{thm-co-1}, $W_{1,\varphi} = C_{\varphi}$ and $W_{1, \phi} = C_{\phi}$ are compact from $\calF^p(\mathbb C^n)$ to $\calF^q(\mathbb C^n)$.

By Proposition \ref{prop-tp-com}, $C_{\varphi} \sim C_{\phi}$ in $\calC(\calF^p(\mathbb C^n), \calF^q(\mathbb C^n))$ via a path in the set $\calC_0(\calF^p(\mathbb C^n), \calF^q(\mathbb C^n))$, and hence, $C_{\varphi} \sim C_{\phi}$ in $\calC_w(\calF^p(\mathbb C^n), \calF^q(\mathbb C^n))$ via a path in $\calC_{w,0}(\calF^p(\mathbb C^n), \calF^q(\mathbb C^n))$. 

On the other hand, by Lemma \ref{lem-tp-psi}, we can see that $W_{\psi,\varphi} \sim C_{\varphi}$ and $W_{\chi, \phi} \sim C_{\phi}$ in $\calC_w(\calF^p(\mathbb C^n), \calF^q(\mathbb C^n))$ via the paths of such type $W_{u_t, \varphi}$, and respectively, $W_{u_t, \phi}$ in $\calC_{w,0}(\calF^p(\mathbb C^n), \calF^q(\mathbb C^n))$. 

From these the desired assertion follows.
\end{proof}

Now for the case $0 < q < p < \infty$, from Proposition \ref{prop-tp-wco-0} we get the following result.

\begin{thm}\label{thm-tp-qp}
Let $0 < q < p <\infty$. The space $\calC_w(\calF^p(\mathbb C^n), \calF^q(\mathbb C^n))$ is path connected.
\end{thm}
\begin{proof}
In view of Proposition \ref{prop-tp-wco-0}, it is enough to prove that 
$$
\calC_w(\calF^p(\mathbb C^n), \calF^q(\mathbb C^n)) = \calC_{w,0}(\calF^p(\mathbb C^n), \calF^q(\mathbb C^n)).
$$
Let $W_{\psi,\varphi} \in \calC_w(\calF^p(\mathbb C^n), \calF^q(\mathbb C^n))$ with $\varphi(z) = Az + b$. If $A = 0$, then $W_{\psi,\varphi} \in \calC_{w,0}(\calF^p(\mathbb C^n), \calF^q(\mathbb C^n))$. Now suppose that $\textnormal{rank}A = s > 0$. 
It suffices to show that $\|A\| < 1$. By contradiction assume that $\|A\| = 1$. Then by Lemma \ref{lem-psiphi}, for the normalization $(\widetilde{\psi}, \widetilde{\varphi})$ of $(\psi, \varphi)$, we have
$$
\widetilde{\psi}(z) = e^{-\left\langle z_{[j]}, \widetilde{b}_{[j]}\right\rangle} \widetilde{\psi}_*(z'_{[j]}), \ z \in \mathbb C^n,
$$
where  $j = \max\{i: \widetilde{a}_{ii} = 1\}$ and $\widetilde{\psi}_*$ is a nonzero entire function of $z'_{[j]}$ on $\mathbb C^{n-j}$. Obviously, $1 \leq j \leq s$.
Then, for every $z \in \mathbb C^n$,
\begin{align*}
\ell_{z_{[s]}}(\widetilde{\psi}, \widetilde{\varphi}) = & \|\widetilde{\psi}(z_{[s]},\cdot)\|_{n-s, q}\; e^{\frac{\left|\widetilde{\varphi}(z)\right|^2 - \left|z_{[s]}\right|^2}{2}}\\
= & \|\widetilde{\psi}_*(z_{j+1},...,z_{s},\cdot)\|_{n-s, q}\; e^{\frac{\left|\widetilde{b}_{[j]}\right|^2}{2}} e^{\frac{\left|\widetilde{A}'_{[j]}z'_{[j]} + \widetilde{b}'_{[j]}\right|^2 - \left|(z_{j+1},...,z_s)\right|^2}{2}},
\end{align*}
by convention that $(z_{j+1},...,z_{s}) = \emptyset$ when $j = s$.
This shows that $\ell_{z_{[s]}}(\widetilde{\psi}, \widetilde{\varphi})$ does not depend on $z_{[j]}$, and hence, it cannot belong to the space $L^{\frac{pq}{p-q}}(\mathbb C^s, dA)$. 

On the other hand, by \cite[Theorem 3.12]{TK-17-1}, $\ell_{z_{[s]}}(\widetilde{\psi}, \widetilde{\varphi}) \in L^{\frac{pq}{p-q}}(\mathbb C^s, dA)$. 
This contradiction completes the proof.
\end{proof}

Next we focus on the case when $0 < p \leq q < \infty$ which is more complicated. We have the following result for the set $C_{w,0}(\calF^p(\mathbb C^n), \calF^q(\mathbb C^n))$.

\begin{prop}\label{prop-cl}
Let $0 < p \leq q < \infty$. The set $C_{w,0}(\calF^p(\mathbb C^n), \calF^q(\mathbb C^n))$ is closed in the space $C_{w}(\calF^p(\mathbb C^n), \calF^q(\mathbb C^n))$.
\end{prop}
\begin{proof}
Take an arbitrary sequence of weighted composition operators $(W_{\psi_i,\varphi_i})_i \subset C_{w,0}(\calF^p(\mathbb C^n), \calF^q(\mathbb C^n))$ converging to $W_{\psi,\varphi}$ in the space $C_{w}(\calF^p(\mathbb C^n), \calF^q(\mathbb C^n))$. Suppose that $\varphi(z) = A z + b$ and $\varphi_i(z) = A^i z + b^i$ with $\|A^i\| < 1$ for all $i \in \mathbb N$. 
It is enough to prove that $\|A\| < 1$. By contradiction assume that $\|A\| = 1$. 

Let $(\widetilde{\psi}, \widetilde{\varphi})$ be the normalization of $(\psi,\varphi)$ with respect to the singular value decomposition $A = V \widetilde{A} U$. Then by \eqref{eq-w-nor},
$$
W_{\psi,\varphi} = C_U W_{\widetilde{\psi},\widetilde{\varphi}} C_V \text{ and } W_{\widetilde{\psi},\widetilde{\varphi}} = C_{U^*}W_{\psi,\varphi}  C_{V^*}.
$$
We also put
$$
\psi_{i,1}(z) = \psi_i(U^*z) \text{ and } \varphi_{i,1}(z) = V^*A^iU^*z + V^*b^i, \ z \in \mathbb C^n,\ i \in \mathbb N.
$$
Then, for each $i \in \mathbb N$, similarly to \eqref{eq-w-nor} we have
$$
W_{\psi_{i},\varphi_{i}} = C_U W_{\psi_{i,1},\varphi_{i,1}} C_V \text{ and } W_{\psi_{i,1},\varphi_{i,1}} = C_{U^*} W_{\psi_{i},\varphi_{i}} C_{V^*}.
$$
From this, all operators $W_{\psi_{i, 1},\varphi_{i,1}}$ belong to $\calC_w(\calF^p(\mathbb C^n), \calF^q(\mathbb C^n))$, hence, by \cite[Proposition 3.1]{TK-17-1}, $m(\psi_{i,1},\varphi_{i,1}) < \infty$; moreover
$$
\|W_{\psi, \varphi} - W_{\psi_i,\varphi_i}\| = \|W_{\widetilde{\psi}, \widetilde{\varphi}} - W_{\psi_{i,1},\varphi_{i,1}}\|,
$$
for every $i \in \mathbb N$.

On the other hand, since $\|A\| = 1$, by Lemma \ref{lem-psiphi}, 
$$
\widetilde{\psi}(z) = e^{-\left\langle z_{[j]}, \widetilde{b}_{[j]} \right\rangle} \widetilde{\psi}_*(z'_{[j]}),
$$ where $j = \max \{i: \widetilde{a}_{ii} = 1\}$ and $ \widetilde{\psi}_*$ is a nonzero entire function of $z'_{[j]}$. 
We divide into two cases of $ \widetilde{\psi}_*$.

\textbf{Case 1.} Suppose that $\widetilde{\psi}_*(0'_{[j]}) \neq 0$. 
Then, for each $i \in \mathbb N$ and $z \in \mathbb C^n$, by Lemma \ref{lem-F1}(i), we have
\begin{align*}
& \|W_{\psi, \varphi} - W_{\psi_i,\varphi_i}\| = \|W_{\widetilde{\psi}, \widetilde{\varphi}} - W_{\psi_{i,1},\varphi_{i,1}}\|  \\
\geq & \|W_{\widetilde{\psi}, \widetilde{\varphi}}k_{\widetilde{\varphi}(z)} - W_{\psi_{i,1},\varphi_{i,1}}k_{\widetilde{\varphi}(z)}\|_{n,q} \\
\geq & \left|\widetilde{\psi}(z) e^{\frac{\left|\widetilde{\varphi}(z)\right|^2 - |z|^2}{2}} \right| - \left|\psi_{i,1}(z) e^{\frac{\left|\varphi_{i,1}(z)\right|^2 - |z|^2}{2}}e^{-\frac{\left|\widetilde{\varphi}(z)-\varphi_{i,1}(z)\right|^2}{2}} \right|\\
\geq & \left|\widetilde{\psi}_*(z'_{[j]})\right| e^{\frac{\left|\widetilde{b}_{[j]}\right|^2}{2}}e^{\frac{\left|\widetilde{A}'_{[j]}z'_{[j]} + \widetilde{b}'_{[j]}\right|^2 - \left|z'_{[j]}\right|^2}{2}} - m(\psi_{i,1},\varphi_{i,1})e^{-\frac{\left|\widetilde{\varphi}(z)-\varphi_{i,1}(z)\right|^2}{2}}.
\end{align*}
In particular, for $z = \lambda \xi$ with $\xi= (1_{[j]},0'_{[j]})$ and $\lambda \in \mathbb C$, the last inequality means that
\begin{align*}
\|W_{\psi, \varphi} - W_{\psi_i,\varphi_i}\|  \geq \left|\widetilde{\psi}_*(0'_{[j]})\right| e^{\frac{\left|\widetilde{b}\right|^2}{2}} - m(\psi_{i,1},\varphi_{i,1})e^{-\frac{\left|\widetilde{\varphi}(\lambda \xi)-\varphi_{i,1}(\lambda \xi)\right|^2}{2}}.
\end{align*}

On the other hand, it is easy to see that $|\widetilde{A}\xi| = |\xi|$ and $|V^*A^iU^*\xi| < |\xi|$, then
$$
|\widetilde{\varphi}(\lambda \xi)-\varphi_{i,1}(\lambda \xi)| \geq |\lambda| (|\xi| - |V^*A^iU^*\xi|) - |\widetilde{b} - V^*b^i| \to + \infty \text{ as } \lambda \to \infty.
$$
Consequently,
$ \|W_{\psi, \varphi} - W_{\psi_i,\varphi_i}\| \geq \left|\widetilde{\psi}_*(0'_{[j]})\right| e^{\frac{\left|\widetilde{b}\right|^2}{2}}$ for all $i \in \mathbb N$.
This means that the sequence $W_{\psi_i,\varphi_i}$ cannot converge to $W_{\psi, \varphi}$ in the space $C_{w}(\calF^p(\mathbb C^n), \calF^q(\mathbb C^n))$, which is a contradiction.

\textbf{Case 2.} Suppose that $\widetilde{\psi}_*((z^0)'_{[j]}) \neq 0$ for some point $z^0 \in \mathbb C^n$ with $(z^0)'_{[j]} \neq 0'_{[j]}$ in $\CC^{n-j}$. We put
$$
\widetilde{\varphi}_0(z_{[j]},z'_{[j]}) = \widetilde{\varphi}\big(z_{[j]},z'_{[j]}+ (z^0)'_{[j]}\big), 
$$
and
\begin{align*}
\widetilde{\psi}_0(z_{[j]},z'_{[j]}) &= \widetilde{\psi}\big(z_{[j]},z'_{[j]} + (z^0)'_{[j]}\big)e^{-\left\langle z'_{[j]},(z^0)'_{[j]} \right\rangle}\\
&= \widetilde{\psi}_*\big(z'_{[j]} + (z^0)'_{[j]}\big) e^{-\left\langle z'_{[j]},(z^0)'_{[j]}\right\rangle -\left\langle z_{[j]}, \widetilde{b}_{[j]} \right\rangle} \\
&= (\widetilde{\psi}_0)_*(z'_{[j]})e^{-\left\langle z_{[j]}, \widetilde{b}_{[j]} \right\rangle},
\end{align*}
where $(\widetilde{\psi}_0)_*(z'_{[j]}) = \widetilde{\psi}_*\big(z'_{[j]} + (z^0)'_{[j]}\big) e^{-\left\langle z'_{[j]},(z^0)'_{[j]}\right\rangle}, \ z \in \mathbb C^n$.

Similarly, for each $i \in \mathbb N$ we define
$$
\varphi_{i,1, 0}(z_{[j]},z'_{[j]}) = \varphi_{i,1}\big(z_{[j]},z'_{[j]}+ (z^0)'_{[j]}\big),
$$
and 
$$
\psi_{i,1,0}(z_{[j]},z'_{[j]}) = \psi_{i,1}\big(z_{[j]},z'_{[j]} + (z^0)'_{[j]}\big)e^{-\left\langle z'_{[j]},(z^0)'_{[j]} \right\rangle}, \ z \in \mathbb C^n. 
$$
Then, all operators $W_{\psi_{i, 1, 0},\varphi_{i, 1, 0}}, \ i \in \mathbb N$, and $W_{\widetilde{\psi}_0,\widetilde{\varphi}_0}$ belong to the space $C_{w}(\calF^p(\mathbb C^n), \calF^q(\mathbb C^n))$. Indeed, we give the proof for $W_{\widetilde{\psi}_0,\widetilde{\varphi}_0}$ (similarly for $W_{\psi_{i, 1, 0},\varphi_{i, 1, 0}}$). For each $f \in \calF^p(\mathbb C^n)$ we have
\begin{align*}
 &\|W_{\widetilde{\psi}_0,\widetilde{\varphi}_0}f\|_{n,q}^q \\ \nonumber
=& \left(\dfrac{q}{2\pi}\right)^{n}\int_{\mathbb C^n} \left|\widetilde{\psi}\big(z_{[j]},z'_{[j]} + (z^0)'_{[j]}\big)e^{-\left\langle z'_{[j]},(z^0)'_{[j]} \right\rangle} f\big(\widetilde{\varphi}\big(z_{[j]},z'_{[j]}+ (z^0)'_{[j]}\big)\big)\right|^q \\ \nonumber
& \quad \quad \quad \quad \quad \quad \times e^{-\frac{q\left|(z_{[j]},z'_{[j]})\right|^2}{2}} dA(z_{[j]},z'_{[j]})\\ \nonumber
=&\; e^{\frac{q\left|(z^0)'_{[j]}\right|^2}{2}} \left(\dfrac{q}{2\pi}\right)^{n}\int_{\mathbb C^n} \left|\widetilde{\psi}\big(z_{[j]},z'_{[j]} + (z^0)'_{[j]}\big)  f\big(\widetilde{\varphi}\big(z_{[j]},z'_{[j]}+ (z^0)'_{[j]}\big)\big)\right|^q \\  \nonumber
& \quad \quad \quad \quad \quad \quad  \quad \quad \quad \times e^{-\frac{q\left|\left(z_{[j]},z'_{[j]} + (z^0)'_{[j]}\right)\right|^2}{2}} dA(z_{[j]},z'_{[j]})\\  \nonumber
=&\; e^{\frac{q\left|(z^0)'_{[j]}\right|^2}{2}} \left(\dfrac{q}{2\pi}\right)^{n}\int_{\mathbb C^n} \left|\widetilde{\psi}(z_{[j]},z'_{[j]})  f\big(\widetilde{\varphi}(z_{[j]},z'_{[j]})\big) \right|^q e^{-\frac{q\left|(z_{[j]},z'_{[j]})\right|^2}{2}} dA(z_{[j]},z'_{[j]})\\  \nonumber
=&\; e^{\frac{q\left|(z^0)'_{[j]}\right|^2}{2}} \|W_{\widetilde{\psi},\widetilde{\varphi}}f\|_{n,q}^q.
\end{align*}
That is, $\|W_{\widetilde{\psi}_0,\widetilde{\varphi}_0}f\|_{n,q} =  e^{\frac{\left|(z^0)'_{[j]}\right|^2}{2}} \|W_{\widetilde{\psi},\widetilde{\varphi}}f\|_{n,q}$ for every $f \in \calF^p(\mathbb C^n)$.

Repeating this argument for $W_{\widetilde{\psi}_0,\widetilde{\varphi}_0} - W_{\psi_{i, 1, 0},\varphi_{i, 1, 0}}$, we can get
\begin{align*}
\|W_{\widetilde{\psi}_0,\widetilde{\varphi}_0}f - W_{\psi_{i, 1, 0},\varphi_{i, 1, 0}}f\|_{n,q} = e^{\frac{\left|(z^0)'_{[j]}\right|^2}{2}} \|W_{\widetilde{\psi},\widetilde{\varphi}}f - W_{\psi_{i, 1},\varphi_{i, 1}}f\|_{n,q},
\end{align*}
for each $f \in \calF^p(\mathbb C^n)$.
It implies that 
$$
\|W_{\widetilde{\psi}_0,\widetilde{\varphi}_0} - W_{\psi_{i, 1, 0},\varphi_{i, 1, 0}}\| = e^{\frac{\left|(z^0)'_{[j]}\right|^2}{2}} \|W_{\widetilde{\psi},\widetilde{\varphi}} - W_{\psi_{i, 1},\varphi_{i, 1}}\|, \ \forall i \in \mathbb N.
$$

Moreover, applying Case 1 to the norm $\|W_{\widetilde{\psi}_0,\widetilde{\varphi}_0} - W_{\psi_{i, 1, 0},\varphi_{i, 1, 0}}\|$ with $(\widetilde{\psi}_0)_*(0'_{[j]}) = \widetilde{\psi}_*((z^0)'_{[j]}) \neq 0$, we have
$$
\|W_{\widetilde{\psi}_0, \widetilde{\varphi}_0} - W_{\psi_{i,1,0},\varphi_{i,1,0}}\| \geq \left|(\widetilde{\psi}_0)_*(0'_{[j]})\right| e^{\frac{\left|\widetilde{b_0}\right|^2}{2}} = \left|\widetilde{\psi}_*((z^0)'_{[j]})\right| e^{\frac{\left|\widetilde{b_0}\right|^2}{2}}, \ \forall i \in \mathbb N,
$$
where $\widetilde{b_0} = \widetilde{\varphi}_0(0) = \widetilde{\varphi}(0_{[j]},(z_0)'_{[j]})$. 
Consequently, for each $i \in \mathbb N$,
\begin{align*}
\|W_{\psi,\varphi} - W_{\psi_{i},\varphi_{i}}\| & = \|W_{\widetilde{\psi},\widetilde{\varphi}} - W_{\psi_{i, 1},\varphi_{i, 1}}\| \\
& = e^{-\frac{\left|(z^0)'_{[j]}\right|^2}{2}} \|W_{\widetilde{\psi}_0,\widetilde{\varphi}_0} - W_{\psi_{i, 1, 0},\varphi_{i, 1, 0}}\| \\
& \geq \left|\widetilde{\psi}_*((z^0)'_{[j]})\right|e^{\frac{\left|\widetilde{b_0}\right|^2}{2}} e^{-\frac{\left|(z^0)'_{[j]}\right|^2}{2}}  > 0.
\end{align*}
It means that the sequence $W_{\psi_i,\varphi_i}$ cannot converge to $W_{\psi, \varphi}$ in the space $C_{w}(\calF^p(\mathbb C^n), \calF^q(\mathbb C^n))$, which is a contradiction.
\end{proof}

Now we fix a nonzero equivalence class $[A] \in [E_n]$. Let $j \in \NN$ and a pair of $n \times n$ unitary matrices $(V,U)$ be defined as in Lemma \ref{lem-sim}. We say that two vectors $b^1$ and $b^2$ in $\mathbb C^n$ are \textit{equivalent by the class $[A]$}, briefly, write $b^1 \sim b^2$ by $[A]$, if $(V^*b^1)_{[j]} = (V^*b^2)_{[j]}$. 
Note that this is an equivalence relation on $\mathbb C^n$ and its definition does not depend on the choice of the pair of unitary matrics $(V,U)$ in Lemma \ref{lem-sim}. Indeed, suppose that $(\widehat{V},\widehat{U})$ is another pair of unitary matrices in Lemma \ref{lem-sim}. Then, $(V, U)$ and $(\widehat{V},\widehat{U})$ are related by \eqref{eq-uv} in Remark \ref{rem-sim}.

From this, it is easy to see that for every vector $b \in \mathbb C^n$
\begin{align*}
\widehat{V}^*b = V_0 (H_1 \oplus ... \oplus H_d \oplus W_1)^* V^* b 
= V_0 (H_1^* \oplus ... \oplus H_d^* \oplus W_1^*)V^*b,
\end{align*}
where, recall that $(V_0, U_0)$ are defined in \eqref{eq-uv0} and for each $1 \leq i \leq d$, $H_i$ is an $n_i \times n_i$ unitary matrix, and $W_1$, $W_2$ are $(n-s) \times (n-s)$ unitary matrices in \eqref{eq-uv} with $n_1 = j$.
Hence, by \eqref{eq-uv}, 
$$
(\widehat{V}^*b)_{[j]} = ((H_1^* \oplus ... \oplus H_d^* \oplus W_1^*)V^*b)_{[j]} = H_1^* (V^*b)_{[j]}.
$$
This relation and the unitary property of $H_1$ imply that
$$
(\widehat{V}^*b^1)_{[j]} = (\widehat{V}^*b^2)_{[j]} \text{ if and only if } (V^*b^1)_{[j]} = (V^*b^2)_{[j]}.
$$
We denote by $[\mathbb C^n]_{[A]}$ the set of all equivalence classes on $\mathbb C^n$ by the above-defined equivalence relation. Since for the zero class $[0] \in [E_n]$, the number $j = 0$, we can suppose that $[\mathbb C^n]_{[0]}$ has only one equivalence class $[0]$ which is the whole space $\mathbb C^n$.

By each $[A] \in [E_n]$ and each $[b] \in [\mathbb C^n]_{[A]}$ we define the following set
$$
\calW([A], [b]) = \{W_{\psi,\varphi}: \varphi(z) = Az + b, A \in [A], b \in [b], \psi \in \calF(n,p,q,\varphi)\}.
$$
Then, by Propositions \ref{prop-tp-wco-0} and \ref{prop-cl}, the set 
\begin{align*}
\calW([0], [0]) &= \{W_{\psi,\varphi}: \varphi(z) = Az + b, A \in [0], b \in \mathbb C^n, \psi \in \calF(n,p,q,\varphi) \}\\
&= \calC_{w,0}(\calF^p(\mathbb C^n), \calF^q(\mathbb C^n))
\end{align*}
is path connected and closed in $\calC_{w}(\calF^p(\mathbb C^n), \calF^q(\mathbb C^n))$.

In view of this, we will prove that this statement is also valid for all sets $\calW([A], [b])$. Then we can get a complete description of all (path) connected components of the space $\calC_{w}(\calF^p(\mathbb C^n), \calF^q(\mathbb C^n))$.

\begin{prop} \label{prop-main-tp-wco}
	Let $ 0 < p \leq q < \infty$.
For each  $[A] \in [E_n]$ and each $[b] \in [\mathbb C^n]_{[A]}$, the set $\calW([A],[b])$ is closed and path connected in $\calC_{w}(\calF^p(\mathbb C^n), \calF^q(\mathbb C^n))$.
\end{prop}
\begin{proof}
By Propositions \ref{prop-tp-wco-0} and \ref{prop-cl}, it suffices to prove for nonzero equivalence classes $[A] \in [E_n]$.
For the reader's convenience we will divide the proof into the following several steps.

Fix a matrix $A^0 \in [A]$ and a vector $b^0 \in [b]$. As in the proof of Lemma \ref{lem-sim}, let $A^0 = V \widetilde{A^0} U$ be the singular value decomposition of $A^0$ with
$
\widetilde{A^0} = \begin{pmatrix}
I_j & 0 \\
0 & G^0
\end{pmatrix},
$
where $I_j$ is the $j \times j$ unit matrix and $G^0$ is a diagonal $(n-j) \times (n-j)$ matrix  with $\|G^0\| < 1$. 

\textbf{Step 1.} We give an explicit representation for all operators in $\calW([A],[b])$.

Fix an arbitrary operator $W_{\psi,\varphi} \in \calW([A],[b])$, i.e. $\varphi(z) = Az + b$ with $A \in [A], b \in [b]$. 
 By Lemma \ref{lem-sim}, we have
$$
A = V A^1 U \text{ with } A^1 = \begin{pmatrix}
I_{j} & 0 \\
0 & G
\end{pmatrix},
$$
where $G$ is an $(n-j) \times (n-j)$ matrix and $\|G\| < 1$.

We put
$$
\psi_1(z) = \psi(U^*z), \ \ \varphi_1(z) = A^1 z + b^1,\ \ b^1 = V^*b.
$$
Similarly to \eqref{eq-w-nor}, we have
$$
W_{\psi,\varphi} = C_U W_{\psi_1,\varphi_1}C_V \text{ and } W_{\psi_1,\varphi_1} = C_{U^*} W_{\psi,\varphi}C_{V^*}.
$$
This shows that the operator $W_{\psi_1,\varphi_1}$ belongs to $\calC_w(\calF^p(\CC^n), \calF^q(\CC^n))$, hence, by \cite[Proposition 3.1]{TK-17-1}, $m(\psi_1,\varphi_1) < \infty$.
 Then applying the argument of Lemma \ref{lem-psiphi} to the pair $(\psi_1,\varphi_1)$, we get that
$\psi_1(z) = e^{-\left\langle z_{[j]}, b^1_{[j]} \right\rangle} \psi_{1,*}(z'_{[j]})$ with  some nonzero entire function $\psi_{1,*}$ of $z'_{[j]}$ on $\mathbb C^{n-j}$. 

Moreover, since $b \in [b]$, $(V^*b)_{[j]} = (V^*b^0)_{[j]}$, hence, $b^1_{[j]} = (V^*b^0)_{[j]}$. Therefore, for each $z \in \mathbb C^n$
\begin{equation*}
\psi_1(z) = e^{-\left\langle z_{[j]}, (V^*b^0)_{[j]} \right\rangle} \psi_{1,*}(z'_{[j]}) \text{ and } \varphi_1(z) = A^1 z + \big((V^*b^0)_{[j]}, (b^1)'_{[j]}\big).
\end{equation*}
Then, for each $f \in \calF^p(\mathbb C^n)$ and $z \in \mathbb C^n$, we have
\begin{align*}
W_{\psi_{1},\varphi_{1}}f(z) &= \psi_{1,*}(z'_{[j]})e^{-\left\langle z_{[j]}, (V^*b^0)_{[j]} \right\rangle} f\big(z_{[j]}+(V^*b^0)_{[j]}, Gz'_{[j]} + (b^1)'_{[j]}\big) \\
& = (T_{(V^*b^0)_{[j]}} W_{\psi_{1,*},\varphi_{1,*}})f(z),
\end{align*}
where 
\begin{equation}\label{eq-phi1}
\varphi_{1,*}(z) = A^1 z + \big(0_{[j]}, (b^1)'_{[j]}\big) = \big(z_{[j]}, Gz'_{[j]} + (b^1)'_{[j]}\big)
\end{equation}
and
\begin{equation}\label{eq-T}
T_{(V^*b^0)_{[j]}}f(z_{[j]}, z'_{[j]}) = e^{-\left\langle z_{[j]},(V^*b^0)_{[j]} \right\rangle} f\big(z_{[j]} + (V^*b^0)_{[j]},z'_{[j]}\big), f \in \calF^p(\mathbb C^n).
\end{equation}
We can check that $T_{(V^*b^0)_{[j]}}$ is, in fact, an invertible bounded weighted composition operator on each space $\calF^p(\mathbb C^n)$ with 
$$
T^{-1}_{(V^*b^0)_{[j]}} = e^{-\left|(V^*b^0)_{[j]}\right|^2} T_{-(V^*b^0)_{[j]}} \text{ and } \|T_{(V^*b^0)_{[j]}}\| = e^{\frac{\left|(V^*b^0)_{[j]}\right|^2}{2}}.
$$

Consequently, each operator $W_{\psi,\varphi}$ in $\calW([A],[b])$ can be represented as follows
$$
W_{\psi, \varphi} = C_U T_{(V^*b^0)_{[j]}} W_{\psi_{1,*},\varphi_{1,*}} C_V,
$$
where $\psi_{1,*}$ is a nonzero entire function of $z'_{[j]}$ and $\varphi_{1,*}$ is of \eqref{eq-phi1} and $T_{(V^*b^0)_{[j]}}$ is defined in \eqref{eq-T}.

\textbf{Step 2.} We show that the set $\calW([A],[b])$ is closed in the space $C_{w}(\calF^p(\mathbb C^{n}), \calF^q(\mathbb C^{n}))$. To do this, we take an arbitrary sequence $(W_{\psi_i, \varphi_i})_i$ from  $\calW([A],[b])$ converging to an operator $W_{\chi, \phi}$ in $C_{w}(\calF^p(\mathbb C^{n}), \calF^q(\mathbb C^{n}))$.
Suppose that $\varphi_i(z) = A^i z + b^i$, $A^i \in [A], b^i \in [b]$ for each $i \in \mathbb N$, and $\phi(z) = Dz + e$. We have to prove that $W_{\chi,\phi}$ is also in $\calW([A],[b])$, that is, $D \in [A]$ and $e \in [b]$.

By Step 1, for each $i \in \mathbb N$,
$$
W_{\psi_i, \varphi_i} = C_U T_{(V^*b^0)_{[j]}} W_{\psi_{i,1,*},\varphi_{i,1,*}} C_V,
$$
where
$\psi_{i,1,*}$ is a nonzero entire function of $z'_{[j]}$ and
$$
\varphi_{i, 1,*}(z) = A^{i,1} z + \big(0_{[j]}, (b^{i,1})'_{[j]}\big) = \big(z_{[j]}, G^i z'_{[j]} + (b^{i,1})'_{[j]}\big),
$$
with $b^{i,1} = V^*b^i$ and 
$
A^{i,1} = \begin{pmatrix}
I_{j} & 0 \\
0 & G^i
\end{pmatrix},
$
where $G^i$ is an $(n-j) \times (n-j)$ matrix with $\|G^i\| < 1$.

Consider the sequence $(W_{\psi_{i,1,*},\varphi_{i,1,*}})_i$. For every $i \in \mathbb N$, obviously, $W_{\psi_{i,1,*},\varphi_{i,1,*}} \in C_{w}(\calF^p(\mathbb C^{n}), \calF^q(\mathbb C^{n}))$, and we have
\begin{align*}
& \| W_{\psi_{i,1,*},\varphi_{i,1,*}} -  T^{-1}_{(V^*b^0)_{[j]}} C_{U^*} W_{\chi,\phi}   C_{V^*}\| \\
= &\; \|T^{-1}_{(V^*b^0)_{[j]}} C_{U^*} W_{\psi_{i},\varphi_{i}}   C_{V^*} - T^{-1}_{(V^*b^0)_{[j]}} C_{U^*} W_{\chi,\phi}   C_{V^*}\|\\
\leq & \; \|T^{-1}_{(V^*b^0)_{[j]}}\| \|W_{\psi_i,\varphi_i} - W_{\chi,\phi} \| \to 0, \text{ as } i \to \infty.
\end{align*}
That is, the sequence $(W_{\psi_{i,1,*},\varphi_{i,1,*}})_i$ converges to the weighted composition operator $W_{\chi_{1,*},\phi_{1,*}} = T^{-1}_{(V^*b^0)_{[j]}} C_{U^*} W_{\chi,\phi}   C_{V^*}$ in $C_{w}(\calF^p(\mathbb C^{n}), \calF^q(\mathbb C^{n}))$. Suppose that $\phi_{1,*}(z) = D^1z + e^1$ with 
$$
D^1 = \begin{pmatrix}
D_{11} & D_{12} \\
D_{21} & D_{22}
\end{pmatrix}
$$
where $D_{11}$, $D_{12}$, $D_{21}$ and $D_{22}$ are $j \times j$, $j \times (n-j)$, $(n-j) \times j$ and $(n-j) \times (n-j)$ matrices, respectively. 
We will show that $D_{11} = I_{j}$, $D_{12} = 0$, $D_{21} = 0$, $\|D_{22}\| < 1$ and $(e^1)_{[j]} = 0$.

For every $f \in \calF^p(\mathbb C^n)$, the sequence $(W_{\psi_{i,1,*},\varphi_{i,1,*}} f)_i$ converges to $W_{\chi_{1,*},\phi_{1,*}}f$ in $\calF^q(\mathbb C^n)$. That is, 
\begin{equation}\label{eq-cl}
\psi_{i,1,*}(z'_{[j]})f\big(z_{[j]}, G^iz'_{[j]} + (b^{i,1})'_{[j]}\big) \to \chi_{1,*}(z) f(D^1z + e^1)
\end{equation}
in $\calF^q(\mathbb C^n)$, and hence, pointwise on $\mathbb C^n$. 

Firstly, with $f(z) \equiv 1$ in $\calF^p(\mathbb C^n)$, \eqref{eq-cl} means that $\psi_{i,1,*}$ pointwise converges to $\chi_{1,*}$ on $\mathbb C^n$. Since $\psi_{i,1,*}$ does not depend on $z_{[j]}$, so does $\chi_{1,*}$, that is, $\chi_{1,*}$ is also a nonzero entire function of $z'_{[j]}$.

Secondly, for every $f \in \calF^p(\mathbb C^n)$ independent on $z'_{[j]}$, \eqref{eq-cl} means that for every $z \in \mathbb C^n$,
$$
\psi_{i,1,*}(z'_{[j]})f(z_{[j]}) \to \chi_{1,*}(z'_{[j]})f\big((D^1z + e^1)_{[j]}\big).
$$
On the other hand, since $\psi_{i,1,*}(z'_{[j]}) \to \chi_{1,*}(z'_{[j]})$, hence 
$$
\psi_{i,1,*}(z'_{[j]})f(z_{[j]}) \to \chi_{1,*}(z'_{[j]})f(z_{[j]}) \text{ for all } z \in \mathbb C^n.
$$
Consequently, $f\big((D^1z + e^1)_{[j]}\big) = f(z_{[j]})$ for all $z \in \mathbb C^n$ and $f \in \calF^p(\mathbb C^n)$ independent on $z'_{[j]}$. From this we can conclude that for every $z \in \mathbb C^n$,
$(D^1z + e^1)_{[j]} = z_{[j]}, \text{ i. e. } D_{11} z_{[j]} + D_{12}z'_{[j]} + (e^1)_{[j]} = z_{[j]}$. Thus, $D_{11} = I_{j}$ and $D_{12} = 0$ and $(e^1)_{[j]} = 0$.

Thirdly, for every $f \in \calF^p(\mathbb C^n)$ independent on $z_{[j]}$, \eqref{eq-cl} means that for every $z \in \mathbb C^n$,
$$
\psi_{i,1,*}(z'_{[j]})f\big(G^i z'_{[j]} + (b^{i,1})'_{[j]}\big) \to \chi_{1,*}(z'_{[j]})f\big((D^1z + e^1)'_{[j]}\big)
$$ 
Since $\psi_{i,1,*}(z'_{[j]})f\big(G^i z'_{[j]} + (b^{i,1})'_{[j]}\big)$ are independent on $z_{[j]}$ for all $i \in \mathbb C^n$, so is $\chi_{1,*}(z'_{[j]})f\big((D^1z + e^1)'_{[j]}\big)$. Therefore, 
$(D^1z + e^1)'_{[j]} = D_{21}z_{[j]} + D_{22} z'_{[j]} + (e^1)'_{[j]}$
is independent on $z_{[j]}$, and hence, $D_{21} = 0$. 

It remains to prove that $\|D_{22}\| < 1$. For the mappings
$$
\varphi_{i,2,*}(z'_{[j]}) = G^i z'_{[j]} + (b^{i,1})'_{[j]} \text{ and } \phi_{2,*}(z'_{[j]}) = D_{22} z'_{[j]} + (e^1)'_{[j]},
$$
the operators $W_{\psi_{i,1,*},\varphi_{i,2,*}}$ and $W_{\chi_{1,*},\phi_{2,*}}$ can be considered as two operators in the space $\calC_w(\calF^p(\mathbb C^{n-j}), \calF^q(\mathbb C^{n-j}))$, by convention that $\CC^{n-j} = \{z'_{[j]} = (z_{j+1},...,z_n): z \in \CC^n\}$ . We show it for $W_{\psi_{i,1,*},\varphi_{i,2,*}}$ (similarly for $W_{\chi_{1,*},\phi_{2,*}}$). For every $f \in \calF^p(\mathbb C^{n-j})$, i. e. for every $f \in \calF^p(\mathbb C^{n})$ independent on $z_{[j]}$, we have $\|f\|_{n-j,p} = \|f\|_{n,p}$ and
\begin{align*}
W_{\psi_{i,1,*},\varphi_{i,2,*}}f(z'_{[j]}) = \psi_{i,1,*}(z'_{[j]})f\big(G^i z'_{[j]} + (b^{i,1})'_{[j]}\big) = W_{\psi_{i,1,*},\varphi_{i,1,*}}f(z),
\end{align*}
hence $\|W_{\psi_{i,1,*},\varphi_{i,2,*}}f\|_{n-j,q} = \|W_{\psi_{i,1,*},\varphi_{i,1,*}}f\|_{n-j,q} = \|W_{\psi_{i,1,*},\varphi_{i,1,*}}f\|_{n,q}$. 
Then, 
\begin{align*}
\|W_{\psi_{i,1,*},\varphi_{i,2,*}}\| &= \sup \{ \|W_{\psi_{i,1,*},\varphi_{i,2,*}}f\|_{n-j,q}: f \in \calF^p(\mathbb C^{n-j}), \|f\|_{n-j,p} \leq 1\}\\
& = \sup \{ \|W_{\psi_{i,1,*},\varphi_{i,1,*}}f\|_{n,q}: f \in \calF^p(\mathbb C^{n-j}), \|f\|_{n,p} \leq 1\} \\
& \leq \sup \{ \|W_{\psi_{i,1,*},\varphi_{i,1,*}}f\|_{n,q}: f \in \calF^p(\mathbb C^{n}), \|f\|_{n,p} \leq 1\} \\
& = \|W_{\psi_{i,1,*},\varphi_{i,1,*}}\|.
\end{align*}
Repeating this argument for $\|W_{\psi_{i,1,*},\varphi_{i,2,*}} - W_{\chi_{1,*},\phi_{2,*}}\|$, we  get that
$
\|W_{\psi_{i,1,*},\varphi_{i,2,*}} - W_{\chi_{1,*},\phi_{2,*}}\|  \leq \|W_{\psi_{i,1,*},\varphi_{i,1,*}} - W_{\chi_{1,*},\phi_{1,*}}\|$ for all $i \in \mathbb N$.
It implies that the sequence $(W_{\psi_{i,1,*},\varphi_{i,2,*}})_i$ converges to $W_{\chi_{1,*},\phi_{2,*}}$ in $C_w(\calF^p(\mathbb C^{n-j}), \calF^q(\mathbb C^{n-j}))$.

On the other hand, since $\|G^i\| < 1$, all operators $W_{\psi_{i,1,*},\varphi_{i,2,*}}$ belong to the set $C_{w,0}(\calF^p(\mathbb C^{n-j}), \calF^q(\mathbb C^{n-j}))$, which, by Proposition \ref{prop-cl}, is closed in $C_w(\calF^p(\mathbb C^{n-j}), \calF^q(\mathbb C^{n-j}))$.
Therefore, $W_{\chi_{1,*},\phi_{2,*}}$ also belongs to the set $C_{w,0}(\calF^p(\mathbb C^{n-j}), \calF^q(\mathbb C^{n-j}))$, and hence, $\|D_{22}\| < 1$.

Thus, we have
$$
W_{\chi,\phi} = C_{U} T_{(V^*b^0)_{[j]}} W_{\chi_{1,*},\phi_{1,*}} C_{V},
$$
where
$\phi_{1,*}(z) = D^1z + e^1$ with $(e^1)_{[j]} = 0$ and
$ D^1 = \begin{pmatrix}
I_j & 0 \\
0 & D_{22}
\end{pmatrix} $
with $\|D_{22}\| < 1$. From this it follows that
$$
\chi(z) = e^{-\left\langle (Uz)_{[j]}, (V^*b^0)_{[j]} \right\rangle} \chi_{1,*}((Uz)'_{[j]})  \text{ and }
\phi(z) = Dz + e,
$$
where
$$
D = V D^1 U = V \begin{pmatrix}
I_{j} & 0 \\
0 & D_{22}
\end{pmatrix} U 
\text{ and } e = V\big((V^*b^0)_{[j]}, (e^1)'_{[j]}\big).
$$
This and Lemma \ref{lem-sim} imply that $D \in [A]$ and $(V^*e)_{[j]} = (V^*b^0)_{[j]}$, and hence, $D \in [A]$ and $e \in [b]$. That is, $W_{\chi,\phi}$ belongs to $\calW([A], [b])$.

\textbf{Step 3.} We prove that the set $\calW([A],[b])$ is path connected. 
Let $W_{\psi,\varphi}$ and $W_{\chi, \phi}$ be two operators in $\calW([A],[b])$ with $\varphi(z) = Az +b$ and $\phi(z) = Dz + e$. Then by Step 1,
$$
W_{\psi,\varphi} = C_U T_{(V^*b^0)_{[j]}} W_{\psi_{1,*},\varphi_{1,*}} C_V
\text{ and }
W_{\chi,\phi} = C_U T_{(V^*b^0)_{[j]}} W_{\chi_{1,*},\phi_{1,*}} C_V,
$$
where
$\psi_{1,*}, \chi_{1,*}$ are nonzero entire functions of $z'_{[j]}$, the function $\varphi_{1,*}$ is of \eqref{eq-phi1} and
$$
\phi_{1,*}(z) = D^1 z + \big(0_{[j]}, (e^1)'_{[j]}\big) = \big(z_{[j]}, Hz'_{[j]} + (e^1)'_{[j]}\big),
$$
with $e^1 = V^*e$ and 
$
D^1 = \begin{pmatrix}
I_{j} & 0 \\
0 & H
\end{pmatrix},
$
where $H$ is an $(n-j) \times (n-j)$ matrix with $\|H\| < 1$.

From this and Theorem \ref{thm-co}, the operators $C_{\varphi_{1,*}}$ and $C_{\phi_{1,*}}$ are bounded from $\calF^p(\mathbb C^n)$ to $\calF^q(\mathbb C^n)$. Then, by Lemma \ref{lem-tp-psi}, $W_{\psi_{1,*},\varphi_{1,*}}\sim C_{\varphi_{1,*}}$ and $W_{\chi_{1,*},\phi_{1,*}} \sim C_{\phi_{1,*}}$ in $\calC_{w}(\calF^p(\mathbb C^{n}),\calF^q(\mathbb C^{n}))$. 

On the other hand, obviously $A^1 \sim D^1$. Then by Theorem \ref{thm-tp-eq-co}, $C_{\varphi_{1,*}} \sim C_{\phi_{1,*}}$ in $\calC(\calF^p(\mathbb C^{n}),\calF^q(\mathbb C^{n}))$, and hence, in $\calC_{w}(\calF^p(\mathbb C^{n}),\calF^q(\mathbb C^{n}))$. So we have $W_{\psi_{1,*},\varphi_{1,*}}\sim C_{\varphi_{1,*}} \sim C_{\phi_{1,*}} \sim W_{\chi_{1,*},\phi_{1,*}}$ in $\calC_{w}(\calF^p(\mathbb C^{n}),\calF^q(\mathbb C^{n}))$.
It means that there is a continuous path $\mathcal{P}_1(t), t \in [0,1],$ connecting $W_{\psi_{1,*},\varphi_{1,*}}$ and $W_{\chi_{1,*},\phi_{1,*}}$ in the space $\calC_{w}(\calF^p(\mathbb C^{n}),\calF^q(\mathbb C^{n}))$. Thus, the map
$$
\mathcal{P}(t) = C_U T_{(V^*b^0)_{[j]}} \mathcal P_1(t) C_V, t \in [0,1],
$$
is a continuous path connecting $W_{\psi,\varphi}$ and $W_{\chi,\phi}$ in $\calC_{w}(\calF^p(\mathbb C^{n}),\calF^q(\mathbb C^{n}))$. Indeed, the operators $C_U$, $T_{(V^*b^0)_{[j]}}$, $C_V$ are nonzero bounded weighted composition operators on every Fock space $\calF^p(\mathbb C^{n})$ and all operators $\mathcal P_1(t), t \in [0,1]$, are in the space $\calC_w(\calF^p(\mathbb C^{n}), \calF^q(\mathbb C^{n}))$. Hence, the operators $\mathcal{P}(t) = C_U T_{(V^*b^0)_{[j]}} \mathcal P_1(t) C_V$ belong to $\calC_{w}(\calF^p(\mathbb C^{n}),\calF^q(\mathbb C^{n}))$ for all $t \in [0,1]$. Moreover, for every $t, t_0 \in [0,1]$ we have
\begin{align*}
 \|\mathcal P(t) - \mathcal P(t_0)\| &=  \|C_U T_{(V^*b^0)_{[j]}} \mathcal P_{1}(t) C_V - C_U T_{(V^*b^0)_{[j]}} \mathcal P_{1}(t_0) C_V\|\\
& \leq \|T_{(V^*b^0)_{[j]}}\| \|\mathcal P_{1}(t) - \mathcal P_{1}(t_0)\| \to 0, \text{ as } t \to t_0.
\end{align*}
Thus, $W_{\psi, \varphi}$ and $W_{\chi,\phi}$ are in the same path connected component of $C_{w}(\calF^p(\mathbb C^{n}), \calF^q(\mathbb C^{n}))$. 

On the other hand, by Step 2, all sets $\calW([A'],[b'])$ with $[A'] \in [E_n]$ and $[b'] \in [\CC^n]_{[A']}$ are closed in the space $\calC_w(\calF^p(\CC^n), \calF^q(\CC^n))$; moreover, they are disjoint. Thus, the path $ \mathcal{P}(t)$ connecting $W_{\psi,\varphi}$ and $W_{\chi,\phi}$ in $\calC_{w}(\calF^p(\mathbb C^{n}),\calF^q(\mathbb C^{n})$ must be in the set $\calW([A],[b])$.

From this Step 3 follows.
\end{proof}

From Propositions \ref{prop-tp-wco-0}, \ref{prop-cl}, \ref{prop-main-tp-wco} we get immediately the following result.

\begin{thm}\label{thm-main-tp}
Let $0 < p \leq q < \infty$, the space $\calC_{w}(\calF^p(\mathbb C^n), \calF^q(\mathbb C^n))$ has the following (path) connected components:
$$
\calC_{w}(\calF^p(\mathbb C^n), \calF^q(\mathbb C^n)) 
= \bigcup_{[A] \in [E_n]} \bigcup_{[b] \in [\mathbb C^n]_{[A]}}\calW([A],[b]).
$$
\end{thm}
\begin{proof}
By Propositions \ref{prop-tp-wco-0}, \ref{prop-cl} and \ref{prop-main-tp-wco}, all sets $\mathcal W([A],[b])$ are path connected and closed in $\calC_{w}(\calF^p(\mathbb C^n), \calF^q(\mathbb C^n))$. Moreover, they are disjoint. Then, each set $\mathcal W([A],[b])$ is a path connected component, and, simultaneously, a connected component in $\calC_w(\calF^p(\mathbb C^n), \calF^q(\mathbb C^n))$.
\end{proof}
\bigskip
 
\textbf{Acknowledgement.} The article was carried out during the first-named author's stay at Division of Mathematical Sciences, School of Physical and Mathematical Sciences, Nanyang Technological University, as a postdoc fellow under the MOE's AcRF Tier 1 M4011724.110 (RG128/16). He would like to thank the institution for hospitality and support.

\end{document}